\documentclass{amsart}

\usepackage{latexsym}
\usepackage{amsmath,amscd,latexsym}
\usepackage[psamsfonts]{amssymb}

\usepackage{amsmath,enumerate}
\usepackage[psamsfonts]{amssymb}
\usepackage{amsthm}
\usepackage[mathscr]{eucal}

\usepackage[ansinew]{inputenc}
\usepackage{colortbl}
\usepackage{color}
\usepackage[active]{srcltx}
\usepackage{graphicx}

\newtheorem{theorem}{Theorem}[section]
\newtheorem{lemma}[theorem]{Lemma}
\newtheorem{corollary}[theorem]{Corollary}
\newtheorem{proposition}[theorem]{Proposition}
\newtheorem{remark}[theorem]{Remark}

\newcommand{\RR}{\mathbb{R}}
\newcommand{\N}{\mathbb{N}}

\newcommand{\D}{{\mathcal D}}

\renewcommand{\u}{{\mathbf u}}
\newcommand{\V}{{\mathcal V}}
\newcommand{\R}{{\mathcal R}}
\newcommand{\G}{{\mathcal G}}
\title[Reaction-diffusion equations with a flux in divergence form]
  {Traveling waves for a Fisher-type reaction-diffusion equation with a flux in divergence form  }
      \thanks{This work was also supported by grants RTI2018-098850-B-I00 from the MINECO-Feder (Spain), and PY18-RT-2422, B-FQM-580-UGR20 and A-FQM-311-UGR18 from the Junta de Andalucia (Spain)}

\author[Arias and Campos ]{Margarita Arias and Juan Campos}

\begin{document}

\begin{abstract}

Analysis of the speed of propagation in  parabolic operators  is frequently carried out considering the minimal speed at which its traveling waves move. This value depends on the solution concept being considered.

We analyze an extensive class of Fisher-type reaction-diffusion equations with flows in divergence form. We work with regular flows, 
which may not meet the standard elliptical conditions, but without other types of singularities.

We show that the range of speeds at which classic traveling waves move is an interval unbounded to the right. Contrary to classic examples, the infimum may not be reached. When the flow is elliptic or over-elliptic, the minimum speed of propagation is achieved.

 The classic traveling wave speed threshold is complemented by another value by analyzing an extension of the first order boundary value problem to which the classic case is reduced.  This singular minimum speed can be justified as a viscous limit of classic minimal speeds in elliptic or over-elliptic flows. 

We construct a singular profile for each speed between  the minimum singular speed and the speeds at which classic traveling waves move. Under additional assumptions, the constructed profile can be justified as that of a traveling wave of the starting equation in the framework of bounded variation functions.

We also show that saturated fronts verifying the Rankine-Hugoniot condition can appear for strictly lower speeds even in the framework of bounded variation functions.
 
\end{abstract}
\maketitle

\section{Introduction}\label{I}

In 1927, Fisher \cite{F} showed that the speed of propagation of the operator  associated with the equation
\begin{equation}\label{F}
\u_t=d \u_{xx} + k\u(1-\u), \qquad d>0\; k>0,
\end{equation}
coincides with the minimal speed at which its traveling waves move, $\sigma^*=2\sqrt{dk}$,  thus justifying what is claimed by  Luther in  1906, see  \cite{L}.
In the wake of this pioneering work, analysis of the speed of propagation in  parabolic operators  is frequently carried out considering the possible speeds at which its traveling waves move. 

 \

Here we consider a reaction-diffusion equation of the form
\begin{equation}\label{rd}
\u_t=(a(\u, \u_x))_x + f(\u),
\end{equation} 
with a flow in divergence form and  a reaction term, $f$, of the logistic type with carrying capacity $1$, so that $\u=0$ and $\u=1$ are constant solutions of \eqref{rd}.

As for the flux, $a$ is  a continuous map on  a relative open subset $\Omega$ of $[0,1]\times \mathbb{R}$. This open set is written as 
 the strip between two maps $\omega_\pm: [0,1]\to [-\infty,\infty]$ with  $ -\infty\leq \omega_-(u)<\omega_+(u)\leq +\infty$, i. e.,
 $$
 \Omega = \{ (u,s)\in \mathbb{R}^2: \, u\in [0,1], \, \omega_-(u)<s< \omega_+(u) \}.
 $$
Saying that  $\Omega$ is an open subset relative to $[0,1]\times \mathbb{R}$  is equivalent to stating
that $\omega_+$ is lower semicontinuous and $\omega_-$ upper semicontinuous. Furthermore, we will always assume symmetry in the second argument,
\begin{equation}\label{nc}
-a(u, s)=a(u, -s), \; u\in[0,1],\;  \omega_-(u)<s< \omega_+(u),
\end{equation}
 and in particular,
$\omega_-(u)=- \omega_+(u)$, for any $u\in [0,1]$. 

The map $s\in(\omega_-(u),\omega_+(u))\to a(u,s)$ will be assumed to be increasing so that the limits
$$
a_\pm(u):=\lim_{s\to \omega_\pm(u)}a(u,s)
$$
always exist. The meaning of  $a_+$ and  $a_-$ is clear: these functions  respectively indicate the maximum flow allowed to the right and left for each concentration level $u$. To avoid some pathological cases we will consider that
\begin{equation}\label{hm}
a_+(u)=\infty, \; \mbox{ when }\; \omega_+(u)<\infty.
\end{equation}
and by symmetry the analogous casa for $a_-$. 

\

In general, a traveling wave, TW from now on,  is a function of the form 
\begin{equation}\label{ondascrecientes}
 \u(t,x)=u(x+\sigma t),
\end{equation}
which is a solution of \eqref{rd} in a sense to be specified. The function $u:\RR\to [0,1]$ is called the profile of the wave and $\sigma \in \RR$ is  the speed at which it moves. We always look for a monotone profile.   A classic TW is obtained by enforcing the profile to be strictly monotone and $C^2(\RR)$.  To be precise we  consider only the increasing case, so that the TW moves to the right when $\sigma>0$, or  to the left when $\sigma<0$.   Other authors usually work with decreasing TWs and look for $\u(t,x)=u(x-\sigma t).$ It is clear that both theories are equivalent. If no growth restrictions are imposed, the same TW can be found moving at opposite speeds. The study of non-monotonous TWs is part of a different theory  and is beyond the scope of this work.

\

The existence of TWs in  \eqref{rd} with constant linear diffusion, $a(u,s)=d s, \; d>0,$ and logistic reaction  was proven almost simultaneously by Fisher \cite{F} and Kolmogorov, Petrosvky and Piscunov \cite{kpp} in their seminal works in 1927. 

They proved that for each  $\sigma\in [2\sqrt{d\,k}, \infty)$, there is a TW  fulfilling the evolution law \eqref{F}
in the classic sense. The corresponding profile connects the two equilibria, $\u=0$ and $\u=1$,  of the associated kinetic equation, $\u_t=k\,\u(1-\u)$. 

The above results extend to more general reaction terms. In particular, for
\begin{equation}\label{1}
\u_t= d\,\u_{xx} + f(u), \qquad d>0, 
\end{equation}
with  $f$  being a regular function so that the corresponding kinetic equation, $\u_t=f(\u)$, is of the logistic type with  carrying capacity $1$, i.e. 
 \begin{equation}\label{l}
 f\in C^1[0,1]\; \mbox{and}\; f(0)=f(1)=0, \; f(u)>0, \; u\in (0,1).
 \end{equation}

 Under these conditions, there exists a positive value, $\sigma^*=\sigma^*(f)$,
 such that  \eqref{1} has a  TW moving at speed $\sigma$, if and only if $\sigma\in [\sigma^*, \infty)$, see \cite{aw1,aw2}.

   The value $\sigma^*$ is not always calculable. Explicit expressions or interesting characterizations even for more general reaction terms can be seen, among others, in \cite{AZ,ACRS,dPV,DPK,GK,MP}. In any event, the lower estimate  
\begin{equation}\label{le}
\sigma^*\geq 2\sqrt{d\,\dot{f}(0)}
\end{equation}
is fulfilled\footnote{ Throughout this work we will denote by $'\,$ $ \frac{d}{d\xi}$ or $ \frac{d}{dt}$, while $\dot{\,}$ will always indicate $\frac{d}{du}$.}. In particular,  TWs do not exist when $\dot{f}(0)=+\infty$. 

Our interest focuses on the flux $a$ rather than on the reaction term. Therefore, from now on we will always assume that $f$ verifies \eqref{l}, although less restrictive conditions can be considered.

\

While classic TWs may appear with more general flows,   when the flow presents degenerations, these classic TWs usually coexist with other types of TWs.  For example, the appearance of so-called {\em sharp-type} TWs --TWs with continuous profiles, not necessarily differentiable, and not strictly monotonous-- is shown in \cite{mm2,SM,EGS, GS,dPV, DT,DT20, AV1}. Furthermore, discontinuous wave profiles may also appear, as in the case of the so-called {\em entropic} TWs, see \cite{ACM,CCCSSinv,CCCSSsurv,CGSS, CS, GarS, CCM}. In all these cases, the concept of solution is intimately linked to the existence of an appropriate theory for the initial values problem.

\

We will state that the flux $a:\Omega\to \RR$ is {\bf regular} if it verifies
\begin{itemize}
\item [$(H_r)$] $\qquad\qquad\qquad\qquad  a\in C^1(\Omega)$ and $ \displaystyle\frac{\partial a}{\partial s}(u, s)>0, \; (u,s) \in \Omega.$
\end{itemize}

When the flux is regular, a classic phase space analysis allows us to prove the existence of a value $\sigma_r$,
so that  \eqref{rd} has a classic TW connecting the two equilibria, $0$ and $1$, and traveling at speed $\sigma$ for each $\sigma>\sigma_r$.  The corresponding profile  is unique (except for  translations) and there are no classic  TWs  traveling at speed $0\leq \sigma<\sigma_r$, see Theorem \ref{tp}.  To do this we use techniques similar to those of, for example,
 \cite{BOO, BS, CCCSSsurv, CCCSSinv, CGSS, CoS, dPV, GarS, GS, mm2, SM}:  we first transform the TW problem in a phase space analysis of a dynamical system and then reduce it to a scalar boundary value problem.
 
 Specifically, taking
\begin{equation}\label{defd}
 \D= \{ (u,v)\in [0,1]\times \RR : a_-(u)<v<a_+(u) \},
\end{equation}
the existence of classic  TWs  is equivalent to finding a solution to the boundary value problem
\begin{equation}\label{contorno}
 \left \{ \begin{array}{ll}
        \dot{V}=\sigma - \frac{f(u)}{g(u,V)}, \; (u,V)\in \D \\
        V(0)=V(1)=0<V(u),\; u\in (0,1),
         \end{array} \right .
\end{equation}
where $g:\D\to \RR$, $s=g(u,V)$ is obtained by solving for $s\in (\omega_-(u), \omega_+(u))$ the equality  $a(u,s)=V$. This is the main objective of Section \ref{RTW}. 

The threshold value $\sigma_r$ is singular. It can either be:  there is a classic TW traveling at speed $\sigma_r$, as in \eqref{F} where $\sigma_r=\sigma^*$;  or  such a classic TW does not exist, see Proposition \ref{r28}. In the latter  case, $\sigma_r$ is not a minimum but  an infimum of the  propagation speeds of the classic TWs. 

The analysis of TWs carried out in  \cite{CCCSSinv,CCCSSsurv}, shows that there can be a non-empty interval of speeds, $[\sigma_{ent},\sigma_{smooth})$, for which the associated profile must be found in the framework of bounded variation functions. The value $\sigma_{smooth}$  there plays  a role similar to that of $\sigma_r$. However, in our case, even imposing the continuity of the functions $a_+$ and $a_-$, see hypothesis $(H_c)$ in Section \ref{STW}, the framework of bounded variation functions may not suffice to give full meaning to the possible profiles.

\

In this work we define a value, $\sigma_s$, equivalent to $\sigma_{ent}$, not associated with the speed of a TW, but analyzing a constant  extension of \eqref{contorno} to all $[0,1]\times \RR$. This is the purpose of Section \ref{STW}.

A first approximation to understand the role of $\sigma_s$ is made in Section \ref{VS}.
We say that  a regular flux $a$  is {\bf  elliptic}  if there are  two constants, $0<k_1<k_2$, so that
\begin{equation} \label{elip}
k_1 s^2 \leq a(u,s)\,s \leq k_2 s^2.
\end{equation}
 If only the inequality on the left holds, that is, if
\begin{equation} \label{oelip}
k_1 s^2 \leq a(u,s)\,s,
\end{equation}
for some $k_1>0$, we  say that  $a$ is  {\bf over-elliptic}. 

In both cases, $a_+(u)=\infty$ for all $u\in [0,1]$, and problem (\ref{contorno}) does not support extension. Therefore,  $\sigma_s=\sigma_r$.

\

Whatever the regular flux $a$, the viscosity approximations of \eqref{rd},
$$
\u_t=(a(\u, \u_x))_x +\varepsilon \u_{xx}+ f(\u),
$$
correspond to fluxes  $a^\varepsilon(u,s)=a(u,s) +\varepsilon s$,  which are always over-elliptic.
Therefore, only a value $ \sigma^\varepsilon=\sigma^\varepsilon_s=\sigma^\varepsilon_r$ is obtained for each flux  $a^\varepsilon$. The main result of  this section, Theorem \ref{v3},  states  $$\sigma^{\varepsilon}\to \sigma_s, \; \varepsilon\to 0.$$
This is the first result that justifies the fact that $\sigma_s$ is more significant than $\sigma_r$.

\

We believe  that when the operator is over-elliptic, its propagation speed is determined by the classic TWs. Note that when there are levels of super-diffusion, that is, there are values  $u\in [0,1]$ where $\omega_+(u)<\infty$, the Cauchy problem associated with equation (\ref{rd}) is not well understood.  Some typical examples under these conditions are covered in  \cite{CS}. 

As   happens in the classic equation \eqref{1},  when $a$ is over-elliptic \eqref{rd} has a classic TW moving at speed  $\sigma_r$. In this case, we can speak of a minimal speed of propagation. See Remark \ref{regularcase}.

\

Sections \ref{A} and \ref{SP} focus on the construction of a profile traveling at speed $\sigma$, when $\sigma$ is between $\sigma_s$ and $\sigma_r$. For this purpose we follow the same line as Theorem \ref{tp}. Additional hypotheses are necessary to give meaning to the solution built in the setting of bounded variation functions, see Theorem \ref{ts}. 

Note that the initial values problem for this kind of operator is far from  being well understood;
 \cite{AnCM} deals with a very similar operator but  without reaction terms and in a bounded interval. We believe that an argument analogous to that of \cite{ACM} can be used in this case, but at the time of writing  this work we are not able to solve some technicalities in  the construction of the required proof.

The concept of entropy solution introduced in \cite{BP, BP2, AnCM} allows us to justify that the obtained singular profile generates a solution of our reaction-diffusion equation. This concept has also recently been used in systems --see \cite{BW, BW2}-- and it is a current research topic.

We dedicate  Section \ref{A} to briefly developing the necessary  theory for this purpose. To do so, we need an additional condition on the growth of the flux, see $(H_g)$ below, which in particular implies  $\omega_+(u)=\infty, \; u\in [0,1]$, and prevents the existence of superdiffusion levels. 

\

All the results thus far presented have been limited to regular fluxes. When there exists some $u_0\in[0,1]$ so that  $a(u_0,s)=0$ for any $s\in \RR$, that is, such that $a_-(u_0)=a_+(u_0)=0$, the regularity of the flux is lost. 
If that happens we will say that level $u_0$ is totally degenerate; and a totally degenerate levels subset can be defined such as
$$L_{td}=\{ u\in [0,1] : a(u,s)=0, \mbox{ for any } s\in \RR\} .$$
As examples of non-regular flows we can consider the flows in \cite{ACM,CCCSSinv,CCCSSsurv,CGSS, CS}, since in all these cases  $0\in L_{td}$. The same
occurs in the doubly nonlinear diffusion operator considered in \cite{AV1,AV2}  when $m>1$. 

\

The existence of totally degenerate levels --and particularly  that $0\in L_{td}$-- is often
related to the appearance of sharp-type TWs with a profile of  half-ray support. This is the case when $a(u,s)=us$, see  \cite{BPU}, 
or in the doubly nonlinear diffusion operator, see \cite{AV1,AV2}. References on this topic can be found in \cite{V2}. The flux-saturated operators \cite{ACM, ACMM, GMP} also present this  phenomenon.

\

Likewise, we obtain an interesting lower estimate  (see Remark \ref{r2}):
$$
\sigma_s\geq 2\sqrt{\frac { \partial a}{\partial s}(0,0)\,\dot{f}(0)}.
$$
 Note that when  $a(u,s)=d s$  this estimate remains \eqref{le}.  

\
Finally, in Appendix \ref{B} we include some technical results on a singular initial value problem that are required throughout this work.

\

\noindent{\bf Some remarks on speed of propagation.}
As  previously stated, in 1906 Luther,  \cite{L}, established that the speed of propagation of a chemical wave is given by a simple formula, $\sigma=2\sqrt{d\,k}$, where $d$ is the  diffusion coefficient and $k$ is a constant depending on the concentration. (See \cite{ST} for a more detailed discussion.) 

However, Luther's idea of speed of propagation was never easy to interpret, as  noted in  \cite{ST}. Having fixed an initial distribution $\u_0(x)\geq 0$,  \eqref{1} has a unique solution $\u$ defined on $(t,x)\in [0,\infty)\times \mathbb{R}$ that satisfies $\u(0,x)=\u_0(x)$, as long as $f$ is a regular function with $f(0)=0$.   The  maximum principle states that  $\u(t,x)>0, \, t>0, \; x\in \mathbb{R}$. We then have  that the speed of propagation is infinite, which is  counterproductive: a chemical substance initially confined in a bottle, after breaking, should   immediately spread to the entire region, however large it might be.

Consider as initial data  a Heaviside function, i.e., $\u_0(x)=0$ if $x<0$ and  $\u_0(x)=1$  if $x\geq 0$.  The solution of \eqref{F}
with $\u(0,x)=\u_0(x)$ --which  satisfies $0<\u(t,x)<1, \, t>0,\, x\in \RR$ thanks to the comparison principle-- immediately becomes a strictly increasing continuous function on $x$ going from 0 to 1 for each $t>0$.
Then,  it is posible to fix a position, $x_c(t)$, such that $\u(t,x_c(t))=c, \, t>0$, for any concentration,  $c\in (0,1)$. In  \cite{kpp} it is proven that
$$
\lim_{t\to \infty} x_c'(t)=2\sqrt{d\,k},
$$
and the level $c$ is not relevant.  The value $\sigma^*=2\sqrt{d\,k}$  therefore defines, in a certain sense, a speed of propagation. 

\

Let us think in terms  of the spread of infections, so common when working with equation \eqref{F}. Let us take a fixed  point $x_0$  on the real line and start moving  at a certain speed $\sigma$. Given  a particular solution $\u(t,x)$ of (\ref{F}), the limit 
\begin{equation}\label{lim1}
\lim_{t\to \infty} \u(t,x_0+\sigma t)
\end{equation}
tells us whether or not we are escaping contagion when we move at that speed. In \cite{aw1,aw2} it is shown that if the initial infection is localized, i.e., if $\u(0,x)$ has  nonempty compact support, this last limit  
is 0 when $|\sigma|>\sigma^*$ and $1$ when  $|\sigma|<\sigma^*$. Therefore, moving faster than $\sigma^*$ we escape infection. Hence, calling such a minimum value {\it speed of escape} makes perfect sense.

\

 In general, we can define this {\it speed of escape}, $\sigma^*$, as follows:
  we consider the values $\sigma\in \RR$ verifying
\begin{itemize}
\item[(P)] whichever is the  solution $u$ of \eqref{rd} with  initial data of compact support and $0\leq u(0,x)\leq 1, \, x\in \RR$, and for any $x_0\in \RR,$ 
$$ \lim_{t\to \infty} \u(t,x_0+\sigma t)=0,$$
\end{itemize}
and we define
\begin{equation}\label{**}
\sigma^*=\inf\{\,|\sigma|: \, \sigma \mbox{ verifies } (P)\}.
\end{equation}

When  working with flux-saturated operators, two levels appear, $\sigma_{smooth}$ and $\sigma_{ent}$, which correspond with $\sigma_r$ and $\sigma_s$, respectively. 
In the framework of \cite{CCCSSsurv, CCCSSinv}, 
a comparison principle like that of \cite{ACM} allows us to show $\sigma^*\leq\sigma_{ent}$. 

It could happen that for a value $\bar{\sigma}<\sigma_{ent}$ there were a TW with an entropic profile. These entropic solutions should probably have a Cantorian singular part in the derivative. Obviously the same comparison principle would assure us that $\sigma^*\leq\bar{\sigma}$.  Although not proven,  such solutions do not seem to exist in this environment. See Remark 3 in \cite{CGSS}. Presumably, here $\sigma^*=\sigma_{ent}$.

\

The definition \eqref{**} in our case has an intrinsic problem: the definition of the solution of the Cauchy problem with initial data of compact support. In addition, this theory would need a good comparison principle of solutions, which is also unknown for this type of operators.
We will address these issues in  future work.

 \section{Classic traveling waves}\label{RTW}
 
 In this section we will assume that the flux in equation \eqref{rd} is regular, i.e., function $a:\Omega\to \RR$ verifies $(H_r)$. We look for classic TWs, that is, solutions of \eqref{rd} of the form $\u(t,x)=u(x+\sigma t)$ for which the profile $u\in C^2(\RR)$. In this case, it is easy to see that the profile  $u$ satisfies the second order ODE
\begin{equation}\label{2o}
 (a(u,u'))'-\sigma u' +f(u)=0,
 \end{equation} 
  for all $\xi\in \RR$. Variable $\xi= x+\sigma t$ stands for the wave coordinate. 
 Note that $\sigma$ is also an unknown of the problem.

\

 As the flux $a$ in \eqref{rd} is defined only in the strip $\Omega$, a wave profile has to satisfy $(u(\xi),  u'(\xi))\in \Omega$ for any 
 $\xi \in \RR$.
It is standard to show that wave profiles of  classic TWs  are monotone solutions  of \eqref{2o}: if $\u(t,x)=u(x+\sigma t)$ is not   the equilibria $\u=0$ or $\u=1$, then $\sigma\neq 0$ and  $u'(\xi)\neq 0, \; \xi \in \RR$.  
 We are going to look for increasing TWs, so that  $u(-\infty)=0$ and $u(+\infty)=1$. A simple integration shows that if an increasing classic TW exists, $\sigma>0$.

 \

Our aim is to show the following

\begin{theorem}\label{tp}
Suppose 
 $(H_r)$, then there exists  $ \sigma_r>0$ such that equation (\ref{rd}) has a classic TW that moves at speed $\sigma$, for each $\sigma >\sigma_r$. Moreover,  there are not classic TWs moving at a speed $\sigma<\sigma_r$.
\end{theorem}

 \begin{remark}\label{rtp}

 Case $\sigma=\sigma_r$ is peculiar. There are examples where \eqref{rd}  will have  a classic TW moving at speed  $\sigma_r$ and others where there is no such solution. As we shall see when $a$ is over-elliptic, there is always a classic TW moving at speed  $\sigma_r$. In Remark \ref{dif} we shall see an example where no classic TWs move at speed $\sigma_r$. \end{remark}

In order to prove Theorem \ref{tp}, let us consider the set
${\D}$ defined in \eqref{defd}.
On the hypotheses of Theorem \ref{tp},  function $g: \, {\D} \to \RR, \; (u,v)\mapsto g(u,v)$, defined by
\begin{equation}\label{gf}a(u, g(u,v))=v,
\end{equation}
is  a $C^1$-function that satisfies $\frac{\partial g}{\partial v}(u, v)>0, \; (u,v) \in {\D}$, and
 (\ref{2o}) is equivalent to the system
\begin{equation}\label{s}
\left\{\begin{array}{l}
u'=g(u,v),\\
v'=\sigma g(u,v) - f(u).
\end{array}\right.
\end{equation}

We look for a solution $(u, v)$ of (\ref{s}) defined on $\RR$ connecting the equilibria $(0,0)$ and $(1,0)$, i.e., satisfying
$$
\lim_{\xi\to -\infty}(u(\xi),v(\xi))=(0,0), \;  \lim_{\xi\to +\infty}(u(\xi),v(\xi))=(1,0), 
$$
and such that
$$
(u(\xi), v(\xi))\in {\D}_+:=\{ (u,v)\in {\D}\,: v>0\}, \; \mbox{ for all } \xi\in \RR.
$$
Note that if $(u, v)$ is a solution of (\ref{s}), while $(u, v)\in {\D}_+$,  the function
$V(u)$ defined by $v(\xi)=V(u(\xi))$ satisfies the first order ODE
\begin{equation}\label{1o}
 \dot{V} = \Psi(u,V,\sigma),
\end{equation}
where $\Psi : \D_+\times [0,+\infty)\to \RR$ is defined by
$$
\Psi(u,V,\sigma):=\sigma - \frac{f(u)}{g(u,V)}.
$$
\begin{proposition}\label{p1}
 On the hypotheses of Theorem \ref{tp}, given  $\sigma \geq 0$  there exist $\alpha_\sigma\in [0,1)$ and $V_\sigma\in C((\alpha_\sigma, 1])\cap C^1((\alpha_\sigma, 1))$  with
 $V_\sigma$ satisfying (\ref{1o}) in $(\alpha_\sigma, 1)$, $V_\sigma(1)=0$ and $V_\sigma(u)>0, \; u\in (\alpha_\sigma, 1)$. Moreover, this function is unique.
\end{proposition}

\proof  Let us denote $R(u)=V(u)^2$ with $V$ being a positive solution of (\ref{1o}). Then,
$$
\dot{R} = 2V\dot{V}= 2\sqrt{R}\left( \sigma - \frac{f(u)}{g(u, \sqrt{R})}\right).
$$
By the hypothesis on $a$, having fixed $u \in [0,1]$ there exists
\begin{equation}\label{rem}
\lim_{v\to 0}\frac{f(u)v}{g(u,v)}= \frac{f(u)}{\frac{\partial g}{\partial v}(u,0)}= f(u)\frac{\partial a}{\partial s}(u,0).
\end{equation}
Hence, taking $\tilde{\Omega}:=\{ (u,R)\in [0,1]\times \RR\, : u\in [0,1], \; -\infty<R<(a_+(u))^2\}$,  the function $\Phi: \tilde{\Omega}\times \RR \to \RR$ defined by
$$
\Phi(u, R;\sigma)=\left\{\begin{array}{ll}
-2f(u){\frac{\partial a}{\partial s}(u,0)}, & R\leq 0, \\ 2\sqrt{R}\left( \sigma - \frac{f(u)}{g(u, \sqrt{R})}\right), & 0<R<(a_+(u))^2,
\end{array}\right.
$$
is continuous  and the Cauchy problem
\begin{equation}\label{p}
\dot{R}=\Phi(u,R;\sigma), \; \; R(1)=0,
\end{equation}
has a solution for each $\sigma \geq 0$. We are going to prove that this solution is unique to the left of 1.

Suppose  $R_i, \; i=1,2$ are two local solutions (\ref{p})  defined on a common interval $(\varepsilon, 1]$, for some $\varepsilon\in [0,1)$, thus
$$
R_1(u)=R_2(u)\qquad u\in (\varepsilon, 1].
$$
Indeed,  take $V_i(u):=\sqrt{R_i(u)}, \; i=1,2,$ and define $\varphi(u):=(V_1(u)-V_2(u))^2. \; \varphi\in C^1(\varepsilon, 1)\cap C(\varepsilon, 1], \; \varphi(u)\geq 0, u\in (\varepsilon, 1]$, and for all $u\in (\varepsilon, 1)$,
$$
\dot{\varphi}(u)= -2f(u) (V_1(u)-V_2(u))\left(\frac{1}{g(u, V_1(u))} - \frac{1}{g(u, V_2(u))}\right) \geq 0,
$$
since function  $\frac{\partial g}{\partial v}(u,v)>0, \; (u,v)\in {\D},$ and $f(u)>0, u \in (0,1)$. Then $\varphi$ is increasing in $(\varepsilon, 1]$. But $\varphi(1)=0$, so $\varphi(u)=0, u\in (\varepsilon, 1]$.

\

The map $R_\sigma :(\alpha_\sigma, 1]\to [0,\infty)$ is now well defined as the maximal solution to the left of problem (\ref{p}). Let us see that $R_\sigma(u)>0, \; u\in(\alpha_\sigma, 1)$.

$R_\sigma(u)\geq 0, \; u\in( \alpha_{\sigma} ,1)$ since if, on the contrary, there is $\bar u\in( \alpha_{\sigma} ,1)$ with $R_\sigma(\bar u)<0$, then  $R_\sigma(u)<0$ for any $u\in (\bar u,1)$, because otherwise,   taking $u_0\in (\bar u,1)$ with $R_\sigma(u_0)=0$ and $R_\sigma(u)<0$ for any $u\in (\bar u, u_0)$,
$$
R(u_0)=R(\bar u) -\int_{\bar u}^{u_0} 2f(\tau){\frac{\partial a}{\partial s}(\tau,0)}\, d\tau <0.
$$
Repeating  the same argument with $u_0=1$ we arrive at a contradiction if $R_\sigma(u)<0, \; u\in (\bar u,1)$. Hence, $R_\sigma(u)\geq 0, \; u\in( \alpha_{\sigma} ,1)$.

Now if  $R_\sigma(\bar u)= 0$, for some $\bar u\in( \alpha_{\sigma} ,1)$, then  $R'(\bar u)=-2f(\bar u){\frac{\partial a}{\partial s}(\bar u,0)}<0$ and $R_\sigma(u)<0, \; u\in (\bar u, \bar u +\delta)$, which is not possible as we have just proven. 

\qed

\begin{proposition}\label{p2} On the hypotheses of Theorem \ref{tp}
 \begin{itemize}
  \item For each $\sigma\geq 0$ there exists
$$
V_\sigma( \alpha_\sigma):= \lim_{u\to \alpha_\sigma^+} V_\sigma(u)\geq a_+(\alpha_\sigma),
$$
and $0\leq V_\sigma( \alpha_\sigma)\leq +\infty.$ 
\item If $0\leq \sigma_1<\sigma_2$, then $\alpha_{\sigma_1}\geq \alpha_{\sigma_2}$ and $V_{\sigma_1}(u) > V_{\sigma_2}(u), \; u\in (\alpha_{\sigma_1}, 1)$.
 \end{itemize}

\end{proposition}
\proof To prove the first assertion, it is enough to observe that
$$\dot{V}_\sigma(u) \leq \sigma , \; u\in (\alpha_\sigma, 1).$$
The 
inequality is a consequence of the results on prolongation of solutions of ODE applied to (\ref{p}).

We will prove the second statement using  the functions $R_\sigma$. Having fixed  $0<\sigma_1<\sigma_2$,  $R_{\sigma_1}$ is a sub-solution of (\ref{p}) with $\sigma=\sigma_2$. Since uniqueness ensures that the solution and sub-solutions are ordered, $\alpha_{\sigma_1}\geq \alpha_{\sigma_2}$ and $R_{\sigma_1}(u) \geq R_{\sigma_2}(u), \; u\in (\alpha_{\sigma_1}, 1)$.  To obtain the  strict inequality note that
 $\dot{R}_{\sigma_1}(u) < \dot{R}_{\sigma_2}(u),$ for any $u\in (\alpha_{\sigma_1}, 1)$ with $R_{\sigma_1}(u) = R_{\sigma_2}(u)$. 
 
 \qed
 
 \
 
 Let us choose $\zeta>0$ such that  $$a_+(u)>\zeta, \; u\in [0,1],$$ which exists since $a_+$ is lower-semicontinuous and $(H_r)$. By (\ref{rem}) we can take $C>0$ with
$$
\frac{2vf(u)}{g(u,v)}< C,\qquad u\in [0,1], v\in (0,\zeta].
$$

\begin{lemma} \label{l1} If $\sigma>\frac{C}{2\zeta}$ then $\alpha_\sigma=0$. Moreover, there exists $\sigma_1>\frac{C}{2\zeta}$ such that
$$
 V_\sigma(0)=0,
$$
for any $\sigma>\sigma_1$. \end{lemma}
\proof Since $R_\sigma(1)=0$, there  exists $0\leq \delta <1$ such that $R_\sigma(u)<\zeta^2, \; u\in [\delta,1]$ and therefore, $\dot{R}_\sigma(u)>2\sigma\sqrt{R_\sigma(u)}-C$, that is, $R_\sigma$ is an upper-solution of equation
\begin{equation}\label{cn}
\dot{R}=2\sigma\sqrt{R}-C
\end{equation}
as long as $R_\sigma(u)<\zeta^2$. But since $R_0=\left(\frac{C}{2\sigma}\right)^2$ is an equilibrium of equation (\ref{cn}) and $R_\sigma(1)=0$,
$$
R_\sigma(u)<\left(\frac{C}{2\sigma}\right)^2, \; u\in [\delta,1].
$$
Therefore, if $\sigma>\frac{C}{2\zeta}$ we have   $\frac{C}{2\sigma}<\zeta$,  and we can set $\delta=0$. Then, $\alpha_\sigma=0$. 

\

To prove the other statement,
take $\eta>0$ large enough for \begin{equation}\label{lim}
\lim_{u\to 0^+} \frac{f(u)}{g(u,\eta u)} = k(\eta)<\infty ,
\end{equation}
to exist.
Note that such a limit exists when  $\eta>0$ is  large enough, because  $g_v(0,0)>0$.  Define
$$\sigma_0:=\max \{\eta+\frac{f(u)}{g(u,\eta u)}: \, u\in (0,1], \, \eta u \leq \zeta \, \} >0,$$
and take $$\sigma_1=\max\{\sigma_0,\frac{C}{2\zeta} \}.$$ Then, if  $\sigma>\sigma_1$,  $\alpha_\sigma=0$. Moreover, $R_{\sigma}(0)=0$, because of
$$
d(u):= (\eta u)^2-R_\sigma(u) >0, \; u\in (0, 1].
$$
Indeed, if there exists $u_0\in (0,1)$ so that $d(u_0)\leq 0$, since $d(1)>0$, we can define
$$
u^*:= \sup \{ u\in (0, 1]: \, d(u)\leq 0\,\}>0.
$$
Then, $d(u^*)=0$ and $\dot{d}(u^*)\geq 0$, but
$$
\dot{d}(u^*)= 2\eta^2 u^* - 2\sqrt{R_\sigma(u^*)} (\sigma - \frac{f(u^*)}{g(u^*, \sqrt{R_\sigma(u^*)})}) = 2\eta u^* (\eta -\sigma + \frac{f(u^*)}{g(u^*,\eta u^*)}).
$$
Yet $(\eta u^*)^2=R_\sigma (u^*)<\zeta^2$, because $\sigma>\sigma_0$. Hence, if $\sigma >\sigma_1, \; \dot{d}(u^*) <0$, which is a contradiction. \qed

\proof[ Proof of Theorem \ref{tp}.]

We  define
\begin{equation}\label{*}
\sigma_r:=\inf \{\sigma >0:\, \alpha_\sigma=0 \mbox{ and } V_\sigma(0)=0 \,\}.
\end{equation}

\

 By Lemma \ref{l1}, $\sigma_r$ is well defined.
 We will first prove that   $\sigma_r>0$. Suppose, by the contrary,  that $\sigma_r=0$ and take $\{\sigma_n\} $ as a  sequence decreasing to 0. Since, by definition, $ V_{\sigma_n}(0)=0$ and $\dot{V}_{\sigma_n}\leq \sigma_n$, we have $$0<V_{\sigma_n}(u)\leq \sigma_n u, \; u\in (0,1).$$ So, $V_{\sigma_n}(u)\to 0,\; u\in (0,1)$ and  therefore, $R_{\sigma_n}(u)\to 0,  \; u\in (0,1)$. Still, continuous dependence, coupled with the uniqueness of the solution, implies that $R_{\sigma_n}(u)\to R_0(u)>0$ for $u\in (\alpha_0,1)$.  Hence, $\sigma_r>0$.

\

Given $\sigma>\sigma_r$,   $V_\sigma$ is defined in $[0,1]$, and $0<V_\sigma(u)<a_+(u), \; u\in(0,1)$. Hence,  we can consider the Cauchy problem

\begin{equation}\label{recu}
\left \{ \begin{array}{l}
          u'=g(u,V_\sigma (u)), \\
          u(0)=u_0,
         \end{array} \right .
\end{equation}
 for each $u_0\in (0,1)$. To solve (\ref{recu}), we define
the function
\begin{equation}\label{G}
G_\sigma(u):=\int_{u_0}^u \frac{d \vartheta}{g(\vartheta, V_\sigma(\vartheta))}, \qquad u\in (0,1).
\end{equation}

It is clear that $G_\sigma \in C^1(0,1)$ and $\dot{G}_\sigma(u) =\frac{1}{g(u, V_\sigma(u))} >0, \; u\in (0,1)$.

\

Let us denote $$\xi_{\sigma}^-=\lim_{u\to 0^+} G_\sigma (u)\qquad \mbox{and} \quad \xi_\sigma^+=\lim_{u\to 1^-} G_\sigma (u).$$
Then, $-\infty \leq \xi_\sigma^-<0<\xi_\sigma^+\leq +\infty$.
The function $u_\sigma: (\xi_\sigma^-, \xi_\sigma^+)\to (0,1)$ defined by
$$
u_\sigma(\xi)=G^{-1}_\sigma(\xi),
$$
is the solution of (\ref{2o}) we are looking for. Indeed, by the implicit function theorem, it is clear that $u_\sigma \in C^1(\xi_\sigma^-, \xi_\sigma^+)$,
$$
u'_\sigma(\xi)= g(u_\sigma(\xi), V_\sigma(u_\sigma (\xi))), \; \xi \in (\xi_\sigma^-, \xi_\sigma^+),
$$
and $u_\sigma(0)=u_0$, i.e., $u_\sigma$ is the solution of (\ref{recu}). Moreover, taking $v_\sigma(\xi)=V_\sigma(u_\sigma (\xi)),$ $v_\sigma \in C^1(\xi_\sigma^-, \xi_\sigma^+)$ and
$$
v'_\sigma(\xi)=\dot{V}_\sigma(u_\sigma (\xi))u'_\sigma(\xi)=\left(\sigma -\frac{f(u_\sigma(\xi))}{g(u_\sigma (\xi), V_\sigma(u_\sigma (\xi)))}\right)g(u_\sigma (\xi), V_\sigma(u_\sigma (\xi)))=
$$
$$
=\sigma g(u_\sigma(\xi), v_\sigma(\xi))- f(u_\sigma(\xi)), \; \xi \in  (\xi_\sigma^-, \xi_\sigma^+),
$$
we find that $(u_\sigma, v_\sigma)$ is a solution of (\ref{s}). By definition, 
$$\lim_{\xi\to \xi_\sigma^-} u_\sigma (\xi)=0\; \mbox{  and }\;
 \lim_{\xi\to \xi_\sigma^+} u_\sigma (\xi)=1.$$
Moreover, $(\xi_\sigma^-, \xi_\sigma^+)=\RR$. Indeed, if, for instance,  $\xi_\sigma^+<+\infty$, since $u'_\sigma(\xi_\sigma^+)=g(1,0)=0$, $u_\sigma $ is the solution on  (\ref{2o}) that satisfies $u(\xi_\sigma^+)=1,\; u'(\xi_\sigma^+)=0$ and, by the uniqueness of solution of the Cauchy problem,   $u_\sigma\equiv 1$, in contradiction with  $u_\sigma(0)=u_0\in (0,  1)$.

\

The last statement  of Theorem \ref{tp} is a consequence of the definition of $\sigma_r$.

\qed

\begin{remark}\label{2.6}
 Note that the  profile of the TW built in Theorem \ref{tp} is unique except for horizontal displacements of the independent variable.
 \end{remark}

The following result will be useful in the next Section.
\begin{lemma} \label{r28} On the hypotheses of Theorem \ref{tp},
equation (\ref{2o}) with $\sigma=\sigma_r$ has a classic TW solution if and only if  $\alpha_{\sigma_r}=0.$
\end{lemma}
 \proof As we have seen in the proof of Theorem \ref{tp},  (\ref{2o}) has a  classic TW solution if and only if  $\alpha_{\sigma}=0$ and $V_{\sigma}(0)=0$. 
 
When  $\alpha_{\sigma_r}=0, \; V_{\sigma_r}(u) $ is defined in $(0,1]$. Moreover, for all $\sigma>\sigma_r, \; V_\sigma(u)\leq \sigma u, \; u\in [0,1]$ since $V_\sigma (0)=0$. Taking limits in this last expression we have  $V_{\sigma_r}(u)\leq \sigma_r u$. Hence, $V_{\sigma_r}(0)=0$.

\qed

\section{The extended first order equation}\label{STW}

In this Section we are going to extend equation (\ref{1o}) to values where $V\geq a_+(u)$. In this way we will define a value $0<\sigma_s\leq \sigma_r$ in a similar way as (\ref{*}) but using solutions of this extension of  \eqref{1o}. To do so, we will need some more regularity in $\D$. Specifically, in addition to $(H_r)$, we will also assume

 \begin{itemize}
 \item [$(H_c)$] The map $a_+ :[0,1]\to [0,\infty]$ is continuous.
\end{itemize}

\

We define
\begin{equation}\label{H}
\mathcal{H}(u,V)=\left\{ \begin{array}{ll}
\frac{1}{g(u,V)}, & 0< V<a_+(u),\\
0, & V\geq a_+(u),
\end{array}\right.
\end{equation}
and 
$\Psi_e :[0,1]\times (0,\infty)\times [0,\infty)\to \RR$
\begin{equation}\label{psi}
\Psi_e(u,V,\sigma)=\sigma-f(u) \mathcal{H}(u,V).
\end{equation}

Now we consider the Cauchy problem
\begin{equation}\label{1p}
\dot{\V}=\Psi_e (u,\V,\sigma),\qquad \V(1)=0.
\end{equation}

\begin{lemma}\label{asHL}
The map $\mathcal{H}:[0,1]\times (0,\infty)\to \RR$ is continuous. Moreover,   \begin{equation} \label{asH}
  \lim_{V\to \infty} \mathcal{H}(u,V)=0,
 \end{equation}
 uniformly on $u\in [0,1]$.
\end{lemma}
\proof Continuity of  $\mathcal{H}$ is equivalent to the following property:

\

 {\it Given any sequence $\{V_n\}_{n\in \N}\subset C( [0,1]) $ that converges uniformly to $V_0$, with $V_0(u)>0$ for $u\in [0,1]$, one has $\mathcal{H}(u,V_n(u))\to \mathcal{H}(u,V_0(u))$ uniformly on $[0,1]$.}

\

Let us see that this property holds.  We can  assume, without loss of generality, that the sequence $\{V_n\}$ is monotone. Having in mind that, fixed $u\in [0,1]$, the map $V\in (0,+\infty)\mapsto \mathcal{H}(u,V)\in \RR$ is continuous, 
$$\mathcal{H}(u,V_n(u))\to \mathcal{H}(u,V_0(u)), \; u\in [0,1].$$Using the continuity of $a_+$, it is clear that  $\mathcal{H}(u,V_n(u))$ is continuous for each $n\in \N$ fixed.  Uniform convergence is then a 
consequence of Dini's Theorem. 

To prove the second statement we proceed in a similar way. We need only to prove (\ref{asH}) for any $u\in [0,1]$; uniform convergence follows from Dini's Theorem. To prove (\ref{asH}) for any $u\in [0,1]$ we can assume that $a_+(u)=\infty$ because otherwise it is evident. But when $a_+(u)=\infty$,
$$\lim_{V\to \infty} g(u,V)=\lim_{s\to \infty} g(u,a(u,s))=\infty$$ since $g(u,a(u,s))=s$ for all $s\in (0,a_+(u))$.  Thus, (\ref{asH}) holds.

\qed

\

Then  $\Psi_e$ continuously  extends function $\Psi$ to all $[0,1]\times (0,\infty)\times [0,\infty)$. The following Lemma improves the first statement of Proposition \ref{p2} and it allows us to extend the function  $V_\sigma$ to a solution of  (\ref{1p}) defined over the entire interval $[0,1]$.
\begin{lemma}\label{l32}
 If $\alpha_\sigma>0$ ,  $V_\sigma(\alpha_\sigma)=a_+(\alpha_\sigma)\in (0,\infty).$
\end{lemma}
\proof Since $a_+$ is continuous and $V_\sigma(u)\leq a_+(u), \; u\in (\alpha_\sigma, 1)$, 
by Proposition \ref{p2}, we have $V_\sigma(\alpha_\sigma)=a_+(\alpha_\sigma)$. That $V_\sigma(\alpha_\sigma)<\infty$ is a consequence of Lemma \ref{asHL}, since $\Psi$ is bounded for  large $V$. 

\qed

\begin{proposition}\label{Ps} 
 For any $\sigma \geq 0$  there is only one function $\V_\sigma\in C([0,1])\cap C^1((0, 1))$ that
 satisfies (\ref{1p}). Moreover,
 \begin{itemize}
  \item $\V_\sigma(u)>0, \; u\in (0, 1)$,
  \item the map $(u,\sigma)\in [0,1]\times[0,+\infty)\mapsto \V_\sigma (u)\in [0,+\infty)$ is continuous, and
  \item for each $u\in (0,1)$, the map $\sigma\in [0,+\infty)\mapsto \V_\sigma (u)\in [0,+\infty)$ is strictly decreasing. 
  \end{itemize}
 \end{proposition}
 
\proof The key idea is to extend (\ref{p}) to all $[0,1]\times \RR\times [0,+\infty)$ in a similar way as $\Psi_e$. We define
\begin{equation}\label{phie}
\Phi_e(u,R,\sigma)=\left\{ \begin{array}{ll}
\Phi(u,R,\sigma), &  R<(a_+(u))^2,\\
2\sigma \sqrt{R}, &  R\geq(a_+(u))^2.
\end{array}\right.
\end{equation}
The continuity of $\mathcal{H}$ guarantees us that   $\Phi_e: [0,1]\times \RR \times [0,+\infty) \to \RR$ is continuous and again the map $\R\to \Phi_e(u,R,\sigma)$ is decreasing when $u$ and $\sigma$ are fixed. Analogously to the proof of Proposition \ref{p1}, we obtain the existence and uniqueness to the left of the solution of the Cauchy problem
\begin{equation}\label{pe}
\dot{\R}=\Phi_e(u,\R,\sigma), \; \R(1)=0,
\end{equation}
 which we denote as $\R_\sigma(u)$. Moreover, $\R_\sigma(u)>0$  where it is defined, and the growth on $\R$ of $ \Phi_e(u,\R,\sigma)$ ensures it is defined on $[0,1]$. Finally, 
  $\V_\sigma(u)=\sqrt{\R_\sigma(u)}$ is the solution sought.
 
 Continuity of the map $(u,\sigma)\in [0,1]\times[0,+\infty)\mapsto \V_\sigma (u)\in [0,+\infty)$ is a consequence of the continuous dependence with respect to $\sigma$ of the solution of \eqref{pe} and monotony is proven as in  Proposition  \ref{p2}.
 
 \qed

Then, it makes sense to define
\begin{equation}\label{sigmas}
 \sigma_s:= \min\{\sigma\geq 0:\, \V_\sigma(0)=0\, \}.
\end{equation}

 \begin{remark} It is evident that  $\sigma_s\leq \sigma_r$. Moreover, $\sigma_s>0$ since $\V_0(u)>0$ in $[0,1)$. This hold since $\dot{\V}_0(u)\leq 0, \; u\in (0,1)$, $\dot{\V}_0(u)< 0, \; u\in (\alpha_0,1)$ and $\V_0(1)=0$.
\end{remark}

\

\begin{proposition} \label{regularwaves}
If equation  (\ref{rd})   has a  classic TW moving at speed  $\sigma_r$, then $\sigma_s=\sigma_r$.
\end{proposition}
\proof As we well know, that equation  (\ref{rd})   has a  classic TW moving at speed  $\sigma_r$ is equivalent to (\ref{2o}) with $\sigma=\sigma_r$ has a  classic solution with $u(-\infty)=0$ and $u(+\infty)=1$. In that case, $\alpha_{\sigma_r}=0$ and $\mathcal{V}_{\sigma_r}(u)=V_{\sigma_r}(u)<a_+(u)$
for all $u\in [0,1]$. Hence, by continuity, there exists $\varepsilon>0$ so that $\mathcal{V}_{\sigma}(u)<a_+(u)$ when $\sigma\in (\sigma_r-\varepsilon,\sigma_r+\varepsilon)$  and $\mathcal{V}_{\sigma}(u)=V_{\sigma}(u)$. If $\sigma_s<\sigma_r$, there would be $\sigma_s <\sigma<\sigma_r$ so that $0=\mathcal{V}_{\sigma}(0)=V_{\sigma}(0)$, in contradiction with the definition of
 $\sigma_r$.
 
 \qed
 \begin{remark}\label{dif} In Remark \ref{ultimoo} we will show an example where $\sigma_s<\sigma_r$. As a consequence of Proposition \ref{regularwaves}, the reaction-diffusion equation constructed does not have a classic TW moving at speed $\sigma_r$, as we claimed in Remark \ref{rtp}.
 \end{remark}

 \begin{remark} \label{regularcase}
Note that when a flux $a$ satisfies $(H_r)$ and  also $a_+(u)=+\infty,$ for all $u\in [0,1]$, condition $(H_c)$ is trivially fulfilled. Furthermore, since in this case the function $\Psi$ does not admit extension, $\sigma_r=\sigma_s$. According to Lemma 
\ref{l32}, $\alpha_\sigma=0$ for any $\sigma\geq 0$ and --as  a consequence of Lemma \ref{r28}-- 
equation \eqref{rd} has a classic TW moving at speed $\sigma$ if and only if $\sigma\geq\sigma_r=\sigma_s$. 
 \end{remark}

\section{Viscosity approximations}\label{VS}

In order to better understand the meaning of $\sigma_s$, in this section  we are going to work with viscosity approximations of the reaction-diffusion equation \eqref{rd}. We will assume throughout the section that flux $a$ satisfies $(H_r)$ and $(H_c)$.

The viscosity equation associated to (\ref{rd}) can be defined as 
\begin{equation}\label{v}
\u_t=(a(\u,\u_x) + \varepsilon \u_x)_x + f(\u),  \;t>0, \; x\in \RR,
\end{equation}
with $\varepsilon$ being a positive parameter.

\

If we denote $a^\varepsilon(u,s):=a(u,s) + \varepsilon s$,  when the flux $a$ is  bounded, the  viscosity approximation $a^\varepsilon$ is  elliptic for all $\varepsilon>0$; but in general we can only ensure $a^\varepsilon$ is over-elliptic. In any case, having fixed $\varepsilon>0$, $a^\varepsilon_+(u)=+\infty, \, u\in [0,1]$ and, 
by  Lemma \ref{regularcase},   $\sigma^\varepsilon_r=\sigma^\varepsilon_s$ for all $\varepsilon >0$, where $\sigma^\varepsilon_r$ and $\sigma^\varepsilon_s$ denote  the values of  $\sigma_r$ and $\sigma_s$ for the flux $a^\varepsilon$. We will keep the notations $\sigma_r$ and $\sigma_s$ for the original flux $a$.

\

Our aim is to prove 
\begin{theorem}\label{v3} Suppose   $(H_r)$ and $(H_c)$. Then
$$
\lim_{\varepsilon \to 0} \sigma^\varepsilon_r=\sigma_s.
$$
\end{theorem}

Throughout   the section we will use  notations similar to those used in the previous ones with the super-index $^\varepsilon$ indicating that we refer to the flux $a^\varepsilon$ instead of to $a$.

\begin{lemma}\label{NV1}
Let $0<\varepsilon_1<\varepsilon_2$ be fixed. Then $\V_\sigma (u)\leq V^{\varepsilon_1}_\sigma(u)\leq V^{\varepsilon_2}_\sigma(u)$ for any $u\in [0,1]$.
 \end{lemma}
 
\proof
We will first show that $V^{\varepsilon_1}_\sigma(u)\leq V^{\varepsilon_2}_\sigma(u)$. By definition,
$$
a^{\varepsilon _1}(u,g^{\varepsilon _1}(u,v))=a^{\varepsilon _2}(u,g^{\varepsilon _2}(u,v)), \; (u,v)\in \RR^2,
$$
that is,
\begin{equation}\label{a}
a(u,g^{\varepsilon _1}(u,v)) - a(u,g^{\varepsilon _2}(u,v)) =\varepsilon _2g^{\varepsilon _2}(u,v)- \varepsilon _1g^{\varepsilon _1}(u,v), \; (u,v)\in \RR^2.
\end{equation}
Since $\frac{\partial a}{\partial s}(u,s) >0, \; (u,s) \in \RR^2$ and $a(u,0)=0, \; u\in \RR$, we have that

$$
g^{\varepsilon _1}(u,v) > g^{\varepsilon _2}(u,v), \; u\in \RR, \; v>0.
$$

Indeed, if $g^{\varepsilon _1}(u,v) \leq g^{\varepsilon _2}(u,v)$, then
$$
a(u,g^{\varepsilon _1}(u,v)) - a(u,g^{\varepsilon _2}(u,v)) \leq 0.
$$
But by definition, for all $u\in\RR, \; i=1,2$,  $g^{\varepsilon _i}(u,v)>0$ when $v>0$, hence
\begin{equation}\label{g}
\varepsilon_1g^{\varepsilon _1}(u,v)<\varepsilon_2g^{\varepsilon _2}(u,v), \; u\in \RR, \; v>0,
\end{equation}
which contradicts (\ref{a}).

Therefore, by (\ref{g}), given $\sigma>0$ fixed,
$$
\Phi^{\varepsilon _1}(u,R;\sigma)>\Phi^{\varepsilon _2}(u,R;\sigma), \; u\in(0,1), \; R>0,
$$
and $R_\sigma^{\varepsilon_1}$ is a upper-solution of the Cauchy problem
$$
\dot{R}=\Phi^{\varepsilon _2}(u,R;\sigma), \; R(1)=0,
$$
so
\begin{equation}\label{R}
R_\sigma^{\varepsilon _1}(u)< R_\sigma^{\varepsilon _2}(u), \; u\in (0, 1).
\end{equation}
 Then $V^{\varepsilon_1}_\sigma(u)\leq V^{\varepsilon_2}_\sigma(u)$ for any $u\in [0,1]$.
 
It only remains to show
$\V_\sigma (u)\leq V^{\varepsilon}_\sigma(u)$ for any $\varepsilon>0$ and $u\in [0,1]$. And this holds since
$$
 \sigma- \frac{f(u)}{g^\varepsilon(u,V)}< \left\{\begin{array}{ll}
 \sigma- \frac{f(u)}{g(u,V)}, & 0<V<a_+(u),\\
 \sigma, &V\geq a_+(u)
 \end{array}\right. 
$$ 
i.e. $\Psi^\varepsilon(u,V,\sigma)<\Psi_e(u,V,\sigma)$, $u \in (0, 1)$ with a similar proof.
  \qed

\proof [Proof of Theorem \ref{v3}]

Using Lemma \ref{NV1}, the map $\varepsilon\to \sigma^\varepsilon_r$ is decreasing and 
$$
 \lim_{\varepsilon\to 0}\sigma^\varepsilon_r\geq \sigma_s.
$$
To prove the conclusion we need to show that 
$$
\lim_{\varepsilon\to 0} V^\varepsilon_\sigma(u)=\V_\sigma (u),
$$
uniformly in $u\in [0,1]$. 
This fact is a consequence of having
$$
\lim_{\varepsilon\to 0} g^\varepsilon(u,V)=\left\{\begin{array}{ll}
 g(u,V), & 0<V<a_+(u),\\
 \infty, &V\geq a_+(u).
 \end{array}\right. 
$$
So 
$$\lim_{\varepsilon\to 0} \Psi^\varepsilon(u,V,\sigma)=\Psi_e(u,V,\sigma),$$
and this limit is uniform in the compact subset of $[0,1]\times \RR$ using an argument as in Lemma \ref{asHL}, which combines the well known Dini Theorem.

\qed

\section{Flux-saturated solutions in the  $BV(\RR)$ framework}\label{A}

In this section we briefly include some general facts related to entropy solutions on the line \cite{BP, BP2, AnCM}. In order not to go into topics related to bounded variation functions in several variables, we will  work only within the framework of TWs.

Given $u \in BV(\RR)$, we will denote as $Du$ its derivative in the sense of the distributions. It is well known that 
$$
Du=u'_{rn}+D_su,
$$
where $u'_{rn}$ is the usually called {\it Radon-Nikodym derivative} of $u$ and it is defined as the absolutely continuous part of $Du$ with respect to the Lebesgue measure in $\RR$, and $D_su$  is the so-called singular part. 
One has that $u'_{rn}\in L^1(\RR)$ and $\|u'_{rn}\|_{L^1}\leq TV(u)$, where $TV(u)$ is the total variation of $u$.

A TW profile $u$ is a  {\it distributional solution} of (\ref{2o}), if $u\in BV(\RR)$ and satisfies
\begin{equation}\label{int}
\int_\RR [a(u(\xi), u'_{rn}(\xi))-\sigma u(\xi)]\varphi'(\xi) \,d\xi=\int_\RR f(u(\xi))\varphi(\xi) \,d\xi,
\end{equation}
for all $\varphi \in \D(\RR)$,  where $\D(\RR)$ represents  Swartz's class of test functions. Note that this definition  requires that $a(u, u_{rn}')\in L^1_{loc}(\RR),$ so it is standard to impose

\

\begin{enumerate}
 \item [$(H_g)$] \qquad 
 $|a(u,s)|\leq \bar a |s|+\tilde a, \;  u\in [0,1], \; s\in \RR$ for some constants $\bar a, \tilde a > 0$.
\end{enumerate}

\

Since $u$ and $f(u)$ are bounded, both sides of the  equality in 
(\ref{int}) make sense for any $\varphi \in \D(\RR)$. Hence, the function 
$$h(\xi):=a(u(\xi), u'_{rn}(\xi))-\sigma u(\xi)$$ belongs to $W^{1,1}_{loc}(\RR)$ and there exists $\hat{h}\in C(\RR)$, such that
$$
\hat{h}(\xi)=h(\xi), \; a.e. \; \xi \in \RR.
$$

\begin{theorem}\label{D} Assume   $(H_r)$ and $(H_g)$ and let $u\in BV(\RR )\cap C^1(\RR \setminus \mathcal {S})$ with $\mathcal {S}$ being a finite set. The following  statements are equivalent:
\begin{enumerate}
\item $u$ is a distributional solution of (\ref{2o}).
\item $u\in C^2(\RR \setminus \mathcal {S})$  satisfies (\ref{2o}) for all $\xi\in \RR \setminus \mathcal {S}$ and the function $$ \hat h(\xi)=a(u(\xi), u'(\xi))-\sigma u(\xi), \; \xi \in \RR  \setminus\mathcal {S},$$
admits a continuous extension to $\RR$, where $u'$ denotes the classic derivative defined in $\RR \setminus \mathcal {S}$.
\end{enumerate}
\end{theorem}

\proof If $u$ is a distributional solution of (\ref{2o}), as $u'_{rn}(\xi)=u'(\xi)$ a.e. $\xi$ in $\RR$, one has that $ h(\xi)=\hat h(\xi)$ a.e. $\xi$ in $\RR$. Then, since $\mathcal{S}$ is finite, $\hat h$  admits a continuous extension to $\RR$ by definition. Moreover, if $I$ is one of the open intervals of $\RR\setminus S$, taking $\varphi \in \D(I)$, since $a(u(\cdot), u'(\cdot)) \in C^1(I)$, we obtain $u\in C^2(I)$,  and it satisfies (\ref{2o}) for all $\xi \in I$.

The reciprocal is a consequence of the following known result:

\

\noindent{\it Let $y\in C(I)\cap C^1(I\setminus \mathcal {S})$ be  a function with $y'\in L^\infty(I)$, where $I$ is a bounded interval of $\RR$ and $\mathcal {S}$ is a finite set. Then, $y\in W^{1,\infty}(I)$ and its weak derivative coincides with its classic derivative.}

\qed

In the framework  of Theorem \ref{D}, we will say that a point $\hat {\xi}\in \mathcal{S}$ is  {\bf saturated} if
 \begin{equation}\label{satt}
 \lim_{\xi\to \hat{\xi}} u'(\xi)= +\infty
\end{equation}
and $$\lim_{\xi\to \hat{\xi}^-}u( \xi):=\mu(\hat{\xi})\leq \lim_{\xi\to \hat{\xi}^+}u(\xi):=\nu(\hat {\xi}).$$
Note that this condition takes into consideration that we look for increasing TWs. Saturated points can also be defined  when $\mu(\hat{\xi})\geq \nu(\hat{\xi})$, changing the derivative condition (\ref{satt}) to be $-\infty$.  The saturation property has already been raised in the literature, see   \cite{ACM, BP, BP2,CCCSSinv,CCCSSsurv, CGSS}. In \cite{BP,RHNC,KR, Bl, Blibro} it can be seen how regular solutions may generate this type of vertical front.

 We will say $u$ is a {\bf flux-saturated profile} of \eqref{2o} if it is a distributional solution of this equation and saturates all the points in $\mathcal {S}$.

Again here the concept is simplified by the fact that $\mathcal{S}$ is finite. When  $u\in BV(\RR)$, the lateral limits that determine $\mu(\hat{\xi})$ and  $\nu(\hat{\xi})$ are always well defined. However,  the sense of the limit in the saturation condition  (\ref{satt}) should be clarified in general, since  $u'_{rn}$ is only defined almost everywhere.

Nor is the meaning of the singular set always clear. When  $u\in BV(\RR)$, in a sense the singular set is defined as the support of $D_su$. If $\mathcal{S}$ is finite and $\hat\xi \in\mathcal{S}$ is a  point of continuity
 of $u$, that is, $\mu(\hat{\xi})= \nu(\hat{\xi})$, then  $\hat\xi\notin supp(D_su)$ where $supp(D_su)$ is the topological support of $D_su$. 
 
 Another thing to keep in mind is that, as a consequence of the definition of distributional solution, we do not need to impose a saturation property
 at a point of continuity  $\hat\xi \in \mathcal{S}$. Hence, if we decompose $D_s u=D_j u+D_cu$ where $D_j$ and $D_c$ denote, respectively,  the jump and the Cantor parts of $D_s$, we are  only imposing a condition on $D_ j u$ because a saturated point $\hat{\xi}\in supp(D_su)$ only if $\mu(\hat{\xi})< \nu(\hat{\xi})$. Theorem \ref{D} and Lemma \ref{cmH4} imply that
$$
a_+(\nu(\hat{\xi})) -\sigma \nu(\hat{\xi}) =a_+(\mu(\hat{\xi})) -\sigma \mu(\hat{\xi}),
$$
that is
\begin{equation}\label{rh}
\sigma= \frac{a_+(\nu(\hat{\xi})) - a_+(\mu(\hat{\xi}))}{\nu(\hat{\xi}) -\mu(\hat{\xi})},
\end{equation}
which is the so-called {\bf Rankine-Hugoniot} condition.

The concept of  {\bf entropy solution} in \cite{ACM, AnCM, CCCSSinv, CCCSSsurv} entails  control of the two parts of $D_su$.
 However in \cite{BP,BP2} the conditions are relaxed and  only  control of  the $D_j$ part is necessary:  at every jump discontinuity point
 $\tilde {\xi}$, 
\begin{equation}\label {BPass}
 \frac{a_+(u) - a_+(\mu(\hat{\xi}))}{u -\mu(\hat{\xi})}\leq \frac{a_+(\nu(\hat{\xi})) - a_+(\mu(\hat{\xi}))}{\nu(\hat{\xi}) -\mu(\hat{\xi})}, \qquad \mbox{ for any } u\in (\mu(\hat{\xi}),\nu(\hat{\xi})],
\end{equation}
must be met.
We will refer to this condition as the  {\bf Bertsch-Dal Passo} condition.

\section{Flux-saturated profiles moving at speed  $\sigma\in [\sigma_s, \sigma_r]$}\label{SP}

In this section we are going  to show the existence of non-classic TWs for (\ref{rd}) that move at speed either $\sigma \in [\sigma_s, \sigma_r)$ or $\sigma=\sigma_r$, when a  classic TW moving at this value does not exist. 

We will assume $(H_r)$ and $(H_c)$, so that, analogously to what was done in Theorem \ref{tp}, when $\sigma\geq \sigma_s$ the TW profile will be  defined by solving  the initial value problem
\begin{equation}\label{pvi}
\left\{\begin{array}{l}u'=g(u,\V_\sigma(u)),\\
u(0)=u_0,
\end{array}\right.
\end{equation}
for some $u_0\in (0,1)$ fixed, where $\V_\sigma$ is the solution of the Cauchy problem (\ref{1p}). To do so, we define $\G_\sigma:(0,1)\to \RR$,
\begin{equation} \label{gsigma}
      \G_\sigma(u):=\int_{u_0}^u \mathcal{H}(\delta,\V_\sigma(\delta))d\delta,        
             \end{equation}

with $\mathcal{H}$ defined in (\ref{H}). 

\begin{lemma} Let $\sigma\geq \sigma_s$ be fixed. Then, 
$\G_\sigma$ is a non-decreasing $C^1$-function with 
 $$\G_\sigma(0)=-\infty, \quad
\G_\sigma(1)=+\infty.$$ 
Moreover, the set of critical leves of $\G_\sigma$ is compact.
\end{lemma}

\proof By definition, 
$$
\dot{\G}_\sigma(u)=\mathcal{H}(u,\V_\sigma(u)), \; u\in (0,1),
$$
so, $\dot{\G}_\sigma(u)\geq 0$ and $\dot{\G}_\sigma(u)>0$ if and only if $\V_\sigma(u)<a_+(u)$. Moreover, since $\sigma\geq\sigma_s$,  $\V_\sigma(0)=\V_\sigma(1)=0$. Therefore there exist $0<\mu<\nu<1$ such that $\dot{\G} _\sigma(u)>0, \; u\in (0,\mu)\cup(\nu,1)$. Then, if we denote $$\mathcal{S_\sigma}=\{\xi\in \RR: \exists \, u\in (0,1)\,\mbox{with }\, \G_\sigma(u)=\xi \,\mbox{and  }\,\dot{\G}_\sigma(u)=0\,\},$$ $ \mathcal{S_\sigma}\subset \G_\sigma( [\mu,\nu])$ and it is a compact set. 

\

Let us see that $\G_\sigma(1)=+\infty$. The proof that  $\G_\sigma(0)=-\infty$ is similar. Let $\bar u\in (\nu,1)$ be fixed and take $v$ as the unique solution of the Cauchy problem
\begin{equation}\label{pvi2}
\left\{ \begin{array}{l}v'=g(v,\V_\sigma(v)),\\
v(0)=\bar u.
\end{array}\right.
\end{equation}
One can prove  that  $v$ is well defined on $[0,\infty)$ and
$v(\xi)\to 1$ when $\xi\to \infty$. Integrating in (\ref{pvi2}) between $0$ and $\xi$ we obtain
$$
\int_{\bar u }^{v(\xi)}\frac{d\delta}{g(\delta,\V_\sigma(\delta))}=\xi,
$$
and taking the limit as $\xi\to \infty$,
$$\int_{\bar u }^{1}\frac{d\delta}{g(\delta,\V_\sigma(\delta))}=\infty.$$
Therefore,
$$
\G_\sigma(u)=\G_\sigma(\bar u)+\int_{\bar u }^{u}\frac{d\delta}{g(\delta,\V_\sigma(\delta))}\to \infty,
$$
as $u\to 1.$

 \qed

Hence, the solution of (\ref{pvi}) will be the function implicitly defined by 
\begin{equation}
      \G_\sigma(u_\sigma(\xi))=\xi, \; \xi\in \RR. \label{imply}
\end{equation}

By Sard's Lemma, the set of critical levels of $\G_\sigma$ has measure  zero  and  (\ref{imply}) defines a single-valued function 
for $\xi \in \RR\setminus \mathcal{S_\sigma}$. 

\begin{theorem}\label{ts} Suppose $(H_r), \; (H_c)$ and $(H_g)$.
Given $\sigma \in [\sigma_s, \sigma_r]$  so that $\mathcal{S_\sigma}$ is a nonempty finite set, 
 the function $u_\sigma$ defined by (\ref{imply}) is a flux-saturated solution of \eqref{2o}. \end{theorem}

As we have just seen in the previous section, we need hypothesis $(H_g)$ in order  to make sense of the concept of flux-saturated solution of \eqref{2o}. Note that in particular this condition implies that $\omega_+(u)=\infty, \; u\in [0,1]$. 

\proof Let $\mathcal{S_\sigma}=\{\xi_1,\xi_2,...,\xi_n\}$ be the set of critical levels of $\G_\sigma$, by the inverse function theorem, (\ref{imply}) defines a unique function, $u_\sigma:\RR\backslash\mathcal{S_\sigma}\to (0,1)$ , that is $C^1(\RR\backslash\mathcal{S_\sigma})$. Using a similar argument to that  used in the proof of Theorem \ref{tp}, we can show that $u_\sigma$ solves (\ref{2o}) in $\RR\backslash\mathcal{S_\sigma}$.

To prove $u_\sigma$ is the required flux-saturated solution we need to show that it is a distributional
solution of this equation that saturates all the points in $\mathcal{S_\sigma}$,  i.e., we have to prove that $$
u_\sigma'(\xi_k)=+\infty, \; k=1,2,\cdots, n,$$ and  the function $$h(\xi)=a(u_\sigma(\xi),u'_\sigma(\xi))-\sigma u_\sigma(\xi), \; \xi\in \RR\backslash\mathcal{S_\sigma},$$ admits a continuous extension to $\RR$. See Theorem \ref{D}.

For any  $k\in \{1,\cdots, n\}$ we denote  $$u_\sigma(\xi_k^-):=\mu_k, \; u_\sigma(\xi_k^+):=\nu_k.$$ 

If $\mu_k<\nu_k$ then $\G_\sigma(u)=\xi_k$ for $u\in [\mu_k, \nu_k]$; and if $\nu_k=\mu_k$ it has to be a critical point. In any case, the first condition follows from applying  the inverse function theorem to (\ref{imply}) since $\dot{\G_\sigma}(\mu_k)=\dot{\G_\sigma}(\nu_k)=0$.

To see the continuity of $h$ we will need the following result, which we will prove below.

 \begin{lemma}  Let $\{(u_p,s_p)\}_{p\in\mathbb{N}} \subset [0,1]\times(0,+\infty)$ be a sequence
 with $s_p<\omega_+(u_p)$, $u_p\to u_0\in [0,1]$ and $s_p\to +\infty$. Then, $a(u_p,s_p)\to a_+(u_0)$ as $p\to \infty$. \label{cmH4}
 \end{lemma}

Since $u_\sigma$ saturates the operator, that is, $u_\sigma'(\xi_k)=+\infty,$ by Lemma \ref{cmH4}, $$h(\xi_k^-)=a_+(\mu_k)-\sigma \mu_k, \; h(\xi_k^+)=a_+(\nu_k)-\sigma \nu_k.$$ So, condition $h(\xi_k^-)=h(\xi_k^+)$ is written as
$$
  a_+(\nu_k)-a_+(\mu_k)=\sigma(\nu_k-\mu_k).
$$
When $\mu_k=\nu_k$ the equality is evident. If $\mu_k<\nu_k$, since $\V_\sigma (u)\geq a_+(u), \; u\in [\mu_k,\nu_k]$, we have $\dot{\V}_\sigma(u)=\sigma, \; u\in[\mu_k,\nu_k]$.  Hence,
$\V_\sigma(\mu_k)-\V_\sigma(\nu_k)=\sigma(\mu_k-\nu_k)$, but,  by construction,
 $$\V_\sigma(\mu_k)=a_+(\mu_k), \; \V_\sigma(\nu_k)=a_+(\nu_k),$$
which gives us the desired equality.

\qed

\
\proof[Proof of Lemma \ref{cmH4}]  Since $s_p<\omega_+(u_p)$, for each $p\in \N$ we have $a(u_p,s_p)<a_+(u_p)$. Suppose there exists a sub-sequence, $\{(u_{p_k},s_{p_k}) \}$ so that $$\lim_k a(u_{p_k},s_{p_k})= l<a_+(u_0).$$ Then, $(u_0, l)\in \mathcal{D}$ and therefore $$s_{p_k}=g(u_{p_k},a(u_{p_k},s_{p_k}))\to g(u_0,l)<\infty,$$
in contradiction with our hypothesis.

 \qed

\begin{proposition}
 The profile $u_\sigma$ obtained in Theorem \ref{ts} satisfies the  Bertsch-Dal Passo condition.
\end{proposition}
\proof
Take $\hat \xi\in \mathcal{S_\sigma}$, and suppose
 $$\mu(\hat{\xi}):=\lim_{\xi\to \hat{\xi}^-}u_\sigma( \xi)< \lim_{\xi\to \hat{\xi}^+}u_\sigma(\xi):=\nu(\hat{\xi}).$$  If $u\in (\mu(\hat{\xi}),\nu(\hat{\xi})]$ then $\G_\sigma(u)= \hat \xi$. Therefore, $\dot{\G}_\sigma(u)=0$ and $\V_\sigma(u)\geq a_+(u)$. 
 
 Hence, for any $u \in (\nu(\hat{\xi}),\mu(\hat{\xi})]$, 
 $$
 \frac{a_+(u) - a_+(\nu(\hat{\xi}))}{u -\nu(\hat{\xi})}\leq \frac{\V_\sigma(u)-\V_\sigma(\nu(\hat{\xi}))}{u-\nu(\hat{\xi})}=\frac{\V_\sigma(\mu(\hat{\xi}))-\V_\sigma(\nu(\hat{\xi}))}{\mu(\hat{\xi})-\nu(\hat{\xi})},
 $$
 since $\dot{\V}_\sigma(u)=\sigma$  for any $u\in [\nu(\hat{\xi}),\mu(\hat{\xi})]$. Because by construction, 
 $$\V_\sigma(\nu(\hat{\xi}))=a_+(\nu(\hat{\xi})), \; \V_\sigma(\mu(\hat{\xi}))=a_+(\mu(\hat{\xi})),$$
 one has that  \eqref{BPass} is verified.
 
\qed

\begin{remark} It seems that the flux-saturated solution of \eqref{2o} verifying the Bertsch-Dal Passo condition is unique; and therefore, in the conditions of Theorem \ref{ts}, the profile of the TW that moves at this speed $\sigma$ has to be the solution of \eqref{imply}.
\end{remark}

To finish this section we will show some criteria that indicate that the framework of Theorem \ref{ts} is quite standard when the flux $a$ satisfies $(H_r), \; (H_c)$ and $(H_g)$. 
\begin{proposition}
Suppose moreover that $a_+$ has derivative at the points where $a_+(u)<+\infty$ and let $\sigma\in[\sigma_s, \sigma_r]$ so that the set
$$
W_\sigma= \{ u\in[0,1]: a_+(u)<+\infty,\; \dot{a}_+(u)=\sigma \}
$$
is  finite, then $\mathcal{S_\sigma}$ is a  finite set.
\end{proposition}
\proof  If  $\mathcal{S_\sigma}$ is nonempty, given $\xi\in \mathcal{S_\sigma}$, since  $\G_\sigma$ is monotone, the set $\{ u\in (0,1) : \G_\sigma (u)=\xi\}$  is necessarily a nonempty compact interval. Let $[\mu,\nu]$ with $0<\mu\leq \nu <1$ be this set. We are going to prove there is a point in $W_\sigma\cap [\mu,\nu]$.

Indeed, when $\mu<\nu$, since  $\dot{\G}_\sigma(u)=0$ for any $u\in (\mu,\nu)$, we have $a_+(u)\leq \V_\sigma(u), \; u\in (\mu,\nu)$. Let us see that 
$a_+(\mu)=\V_\sigma (\mu)$ and  $a_+(\nu)=\V_\sigma (\nu)$.  If, for instance, $a_+(\mu)<\V_\sigma (\mu)$, by continuity we can find an interval $(\mu-\varepsilon,\mu)$ so that $a_+(u)<\V_\sigma (u)$ for any $u\in (\mu-\varepsilon,\mu)$ and then $\dot{\G}_\sigma(u)=0$ for any $u\in (\mu-\varepsilon,\mu)$. Therefore  $\G_\sigma(u)=\xi$ for all  $u\in (\mu-\varepsilon,\mu]$, in contradiction with $[\mu,\nu]=\{ u\in (0,1) : \G_\sigma (u)=\xi\}$.

Hence, for any $u\in [\mu,\nu]$, $\dot{\V}_\sigma(u)=\sigma$ since $a_+(u)\leq \V_\sigma(u)$. Then 
$$
\frac{\V_\sigma (\nu)-\V_\sigma (\mu)}{\nu-\mu}=\sigma,
$$
that is,  $$\sigma=\frac{a_+ (\nu)-a_+ (\mu)}{\nu-\mu}.$$ The conclusion follows from the  Mean Value Theorem. 

\

The case $\nu=\mu$ is a bit more complicated. Since $a_+(\mu)=\V_\sigma (\mu), \; \dot{\V}_\sigma(\mu)=\sigma$. If, for instance,  $\dot{a}_+ (\mu)> \sigma$, there exists $\varepsilon>0$ so that $a_+(u)>\V_\sigma (u)$ for any $u\in (\mu-\varepsilon,\mu)$. Hence, $\dot{\G}_\sigma(u)=0$ for any $u\in (\mu-\varepsilon,\mu]$ and  $\G_\sigma(u)=\xi$ for all  $u\in (\mu-\varepsilon,\mu]$, in contradiction with $[\mu,\nu]=\{ u\in (0,1) : \G_\sigma (u)=\xi\}$. Analogously, we have that $\dot{a}_+ (\mu)\nless \sigma$, so $\dot{a}_+ (\mu)=\sigma$.

\qed

\

We also have the following result:

\begin{theorem}\label{rst2}  Suppose $(H_r), \; (H_c),\; (H_g)$ and that $a_+$ is  convex in $[0,1]$. Then for any  $\sigma\geq\sigma_s$, the set $\mathcal{S_\sigma}$ is either empty or a singleton. As a consequence, there  is always a classic or a flux-saturated TW moving at speed $\sigma$ for any  $\sigma\geq\sigma_s$.
\end{theorem}

\proof  Take $U:=\{ u\in[0,1]: a_+(u)>\V_\sigma(u) \}$.  We are going to show that if $U\neq [0,1]$ then
$U=[0, \mu)\cup (\nu,1]$ for some  $0<\mu\leq \nu<1$. 

Suppose $U\neq [0,1]$ and consider the 
 two auxiliary functions
$$
\begin{array}{l}
\varphi_a(u)=a_+(u)-\sigma u, \; u\in [0,1],  \\ \varphi_v(u)=\V_\sigma(u)-\sigma u, \; u\in [0,1],
\end{array}
$$
so that  $
U=\{ u\in [0,1] :\,\varphi_a(u)>\varphi_v(u) \}.
$ Now define 
$$
\begin{array}{l}
\mu:=\min \{u\in(0,1): \varphi_a(u)\leq \varphi_v(u) \},\\
\nu:=\max \{u\in(0,1):\varphi_a(u)\leq \varphi_v(u) \},
\end{array}
$$
since 
$$
  \begin{array}{l}
 \varphi_a(0)=a_+(0)>0=\varphi_v(0),\\ \varphi_a(1)=a_+(1)-\sigma>-\sigma=\varphi_v(1),
 \end{array}
 $$
 we have $0<\mu\leq \nu <1$. We only need to show that  
$$
\varphi_a(u)\leq \varphi_v(u), \; u\in [\mu, \nu].
$$

 Since
$\dot{\V}_\sigma(u)\leq \sigma$, the function $\varphi_v$ is non-increasing. Therefore, $\varphi_v(\mu)\geq \varphi_v(\nu)$ and, since by construction $\varphi_v(\mu)=\varphi_a(\mu)$ and $ \varphi_v(\nu)=\varphi_a(\nu)$, we have $\varphi_a(\mu)\geq\varphi_a(\nu)$.  If $\varphi_a(\mu)=\varphi_a(\nu)$,   $\varphi_v$ is constant and the result is a consequence of the convexity of $\varphi_a$. 

Suppose then  $\varphi_a(\mu)>\varphi_a(\nu)$. Since $\varphi_a$ is convex, $\varphi_v(\mu)=\varphi_a(\mu)$ and $\dot{\varphi}_v(\mu)=0$ because $\dot{\V}_\sigma(u)=\sigma$ as long as $\V_\sigma(u)\geq a_+(u)$, hence there exists $\varepsilon >0$ so that   $\varphi_a(u)\leq\varphi_v(u), \; u\in [\mu, \mu+\varepsilon]$. But in this case, $\V_\sigma(u)\geq a_+(u), \; u\in [\mu, \mu+\varepsilon]$ and $\dot{\varphi}_v(u)=0$. Consequently, $\varphi_v$ is constant in $[\mu, \nu]$, in contradiction with 
 $ \varphi_v(\nu)=\varphi_a(\nu)$.
 
 \qed

\section{The ultra degenerate case}\label{EX}

The main aim of this section is to obtain an example  where equation (\ref{2o}) has a flux-saturated solution for a value $\sigma=\bar \sigma$ with $0<\bar \sigma<\sigma_s$. The example we propose satisfies the hypotheses $(H_r)$, $(H_c)$ and $(H_g)$, so the regularity of the functions $a$ and $f$ does not seem to be relevant. The built flux-saturated profile does not meet the Bertsch-Dal Passo condition  --this seems to be the main question.

\

To arrive at to the desired example, we need to extend the work environment and allow  \emph{totally degenerate levels}. As we said in the Introduction, they are levels where $a(u,s)=0$ for all $s\in \RR$ and we denote as $L_{td}$ the set of totally degenerate levels of our flux.

\subsection{The ultra degenerate framework}\label{ultraframe}

For the sake of simplicity we will consider $\Omega= [0,1]\times \RR$, so that $a\in C^1([0,1]\times \RR)$. With respect to the reaction term, $f$, we will consider, as usual, that it satisfies \eqref{l}. In addition to  $(H_c)$, we will assume
 \begin{itemize}
 \item[$(C1)$] The totally degenerate levels set, $L_{td}$, is the union of a finite number of intervals.
 \item[$(C2)$] For any $u\notin L_{td}$  and $s\in \RR$, $\frac{\partial a}{\partial s}(u,s)>0$.
 \item[$(C3)$] $a(u,0)=0$ for any $u\in [0,1]$ and there exists $M>0$ such  that $|\frac{\partial a}{\partial s}(u,s)|\leq M, \; \;  u\in [0,1], \; s\in \RR$.
\end{itemize}
 
 As a consequence of $(C2)$, the maps $a_+$ and $a_-$ are well defined and $a_+(u)=a_-(u)=0$ for any $u\in L_{td}$.  
 
 $(C1)$ ensures that the number of times that the function $a_+$ goes from being positive to zero is finite and avoids pathological cases. As for $(C3)$,  it is a condition of growth of  $a(u,s)$ with respect to $s$ stronger than $(H_g)$  that we will need later. 
 
 Under these conditions, the set  $\mathcal{D}=\{(u,v)\in ([0,1]\setminus L_{td})\times \RR :a_-(u)<v<a_+(u)\}$ is still an open subset relative to $[0,1]\times \RR$ and function $g$ is  well defined on $\mathcal{D}$.

\

As in the regular case, the question is to solve \eqref{pvi} for some $u_0\in (0,1)$
where $\V:[0,1]\to \RR$ is a continuous function, which is a {\it formal} solution of \eqref{1p} with $\V(0)=\V(1)=0, \; \V(u)\geq 0, \; u\in (0,1)$.
To construct this  {\it formal} solution  we follow the same guidelines as in the regular case. Hence,  we consider the function $\R(u)=\V(u)^2$, which must satisfy the extended problem \eqref{pe}, 
in addition to $\R(0)=0$ and $\R(u)\geq 0, \; u\in [0,1]$.

The  equation in \eqref{pe} must be understood in the  Carath\'eodory sense. See \cite{CL} for a precise definition.

\begin{lemma}\label{ls} For each $\sigma \geq 0$,  the initial value problem \eqref{pe}
has a unique solution in the sense of Carath\'eodory, $\R_\sigma$. Moreover, $\R_\sigma\in C([0,1])$ and  $\R_\sigma(u)\geq 0, \; u\in [0,1]$.
\end{lemma}

\proof First we are going to check that $\Phi_e$ satisfies  Carath\'eodory conditions for each $\sigma\geq 0$ fixed, see \cite{CL}. Since $\frac{\partial a}{\partial s}(u,0)=0$  when $u\in L_{td}$, $(C3)$ ensures that  the function 
$$
\R\in \RR \mapsto \Phi_e(u,\R;\sigma)\in \RR
$$
is continuous  for any $u\in [0,1]$. On the other hand, 
$$
g(u,v)\geq \frac 1 M v, \; u\in [0,1]\setminus L_{td}, \; 0<v<a_+(u).
$$
Therefore, taking $\tilde M\geq \max\{ f(u)M, f(u)\frac{\partial a}{\partial s}(u,0), \; u\in [0,1]\}$, we have
$$
|\Phi_e(u,\R;\sigma)|\leq \sigma\sqrt{|\R|}+\tilde M, \; u\in [0,1],\; \R\in \RR.
$$
Theorem 1.1, Chapter 2, in \cite{CL} allows us to affirm  the existence of a  solution of \eqref{pe} in the sense of Carath\'eodory. Furthermore, the sub-linear growth in $\R$ ensures that it can be extended to the entire interval $[0,1]$. Note that a Carath\'eodory solution means a continuous function $\R\in C[0, 1]$ that satisfies

$$
\R(u)=-\int_u^1 \Phi_e(\tau,\R(\tau);\sigma)\, d\tau, \; u\in [0, 1].
$$

In particular, if $u_1<u_2$, $$\R(u_2)=\R(u_1)+\int_{u_1}^{u_2} \Phi_e(\tau,\R(\tau);\sigma)\, d\tau.$$ An argument similar to the one  made in the proof of Proposition \ref{p1} shows that if $\R(u_0)<0$ for some $0< u_0<1$, then  $\R(u)\leq \R(u_0), \; u\in [u_0,1]$.  In particular, we would have $\R(1)<0$, which is a contradiction. Then,  $\R(u)\geq 0, \; u\in [0,1]$.

 $\Phi_e(\cdot,\cdot;\sigma)$  is obviously continuous in ${\rm int} (L_{td})\times  \RR$.   Arguments similar to those used in the previous sections  based on Lemma \ref{asHL} allow us to show the continuity of this function when  $u\in[0, 1]\backslash L_{td}$. Therefore, $\R$ is $C^1$ in $[0, 1]\backslash F$,  where $F$  is the finite set formed by the end points of the intervals in $L_{td}$, and so, it is a classic solution of the differential equation in \eqref{pe} in $[0, 1]\backslash F$. Moreover, since $\R(u)=0$ implies $\dot{\R}(u)<0$, $\R(u)>0$ for any $u\in (0,1)\backslash L_{td}$.

\

To finish the proof, let us see that the solution of  \eqref{pe} is unique. Suppose  that  $\R_1$ and $\R_2$ are two solutions of \eqref{pe} and consider  the function
$$
\varphi(u)=\left(\sqrt{\R_1(u)}-\sqrt{\R_2(u)}\right)^2.
$$
$\varphi\in C([0,1])$ and $\varphi(u)\geq 0, \, u\in [0,1]$. As $\varphi(1)=0$ we only have to show that  $\dot\varphi (u)\geq 0$ when $u\in (0,1)$ except for a finite number of points. 

Suppose first that $u\notin L_{td}$, whereby  $\R_i(u)>0, i=1,2$. Since the roles of $\R_1$ and $\R_2$ can obviously be interchanged, we have three possible cases:
\begin{itemize}
\item $\R_i(u)<a_+(u), i=1,2$. In this case,
$$
\dot{\varphi}(u)=-2f(u)\left (\sqrt{\R_1(u)}-\sqrt{\R_2(u)}\right ) \left(\frac{1}{g(u, \sqrt{\R_1(u)})}-\frac{1}{g(u, \sqrt{\R_2(u)})}\right)>0,
$$
since $g(u,v)$ is increasing its second variable.
\item $\R_1(u)<a_+(u)\leq \R_2(u)$, then
$$
\dot{\varphi}(u)=2\left (\sqrt{\R_2(u)}-\sqrt{\R_1(u)}\right ) \frac{f(u)}{g(u, \sqrt{\R_1(u)})}> 0.
$$

\item $a_+(u)\leq \R_i(u), \; i=1,2$, in which case 
$$
\dot{\varphi}(u)= 0.
$$

\end{itemize}
The case $u\in L_{td}$ is more delicate. When $u\in L_{td}\backslash F$, then $u\in (\mu,\nu)\subset L_{td}$ and  $$\dot{\R}_i(u)=2\sigma\sqrt{\R_i(u)}, \; u\in (\mu,\nu), \; i=1,2.$$ Since  equation $\dot{\R}=2\sigma\sqrt{\R}$ has uniqueness on the right, there are several possibilities.   We can assume the existence of $u_1\leq u_2\in [\mu,\nu]$
 with 
$$
\left \{\begin{array}{ll}
\R_1(u)=0=\R_2(u), & \mbox{ if } \mu<u<u_1,  \\
 \R_1(u)>0=\R_2(u), &   \mbox{ if } u_1<u<u_2,\\
 \R_1(u)>0, \R_2(u)>0,  &  \mbox{ if } u_2<u<\nu .
\end{array}\right .
$$
Therefore,
$$
\dot{\varphi}(u)=\left \{\begin{array}{ll}
0 & \mbox{ If } \mu<u<u_1,  \\
2\sigma \sqrt{\R_1(u)}  & \mbox{ If } u_1<u<u_2, \\
0  & \mbox{ If } u_2<u<\nu. 
\end{array}\right .
$$
So,  there exists $\dot\varphi (u)$ and $\dot\varphi (u)\geq 0$ except at a maximum of four points for each interval of $L_{td}$.

We have shown that \eqref{pe} has only one solution in the sense of Carath\'eodory, which we will denote  as $\R_\sigma$.

\qed

\

Now, for each $\sigma\geq 0$, we define the
 function, $\V_\sigma:[0, 1]\to [0,\infty)$  as
 $$\V_\sigma(u)= \sqrt{\R_\sigma(u)},
 $$
 with $\R_\sigma$ the only solution of (\ref{pe}) given by the previous Lemma.  $\V_\sigma$ is continuous, 
 $\V_\sigma (1)=0$, and for any $u\in (0,1)$ with $\V_\sigma (u)>0$, it verifies
\begin{equation}\label{vsigma}
\dot{\V}_\sigma(u)=\left\{\begin{array}{ll}
 \sigma - \frac{f(u)}{g(u, \V_\sigma(u))}, & 0<\V_\sigma(u)<a_+(u), \\
\sigma, & \V_{\sigma}(u)>a_+(u).
\end{array}\right.
\end{equation}

Furthermore,  it can be proven that the map $(\sigma,u)\in[0,+\infty)\times [0,1]\mapsto \V_\sigma(u)\in [0,+\infty)$ is continuous and
decreasing on $\sigma$  for each $u\in [0, 1]$. To do this, we use function $\R_\sigma(u)$  again. The uniqueness of solution of the Cauchy problem (\ref{pe})  is essential.

\

It is not clear, but it appears to be true, probably with some additional condition, that $\V_\sigma(0)=0$ if $\sigma$ is large enough. In any case, if $a$ is non-trivial, that is, if  $L_{td}\ne [0,1]$, we can check that $\V_0(0)>0$. So, we can define 
\begin{equation}\label{sigmas+}\sigma_s:=\min\{\sigma>0: \V_\sigma(0)=0\,\}>0,
\end{equation}
but it could be infinite, and  $\V_\sigma(0)=0$ for all $\sigma\geq \sigma_s$. 

\

When $\sigma\geq \sigma_s$, we can give a formal sense to $\V_\sigma(u)$. The question is to solve   \eqref{imply} with $\G_\sigma $  defined as \eqref{gsigma}. Care must be taken with the function  $\mathcal{H}$ defined in \eqref{H}, since for $u\notin L_{td}$, $\mathcal{H}(u,0)=\infty$. Although we have shown that $\V_\sigma(u)>0$ for all $u\in (0,1)\backslash L_{td}$, we could have integrability problems. Note that except for a finite set, $f(u)\mathcal{H}(u,\V_\sigma(u))=\dot{\V}_\sigma(u)-\sigma$,  and the last expression has a bounded primitive, so in compact intervals of $[0,1]$, $\mathcal{H}(u,\V_\sigma(u))$ is integrable.  

  $\G_\sigma$ is constant in any interval  $[\mu,\nu]\subset  L_{td}$ since $\dot{\G}_\sigma(u)=0, \; u\in (\mu,\nu)$. Adding the critical levels of $\G_\sigma$ outside of $ L_{td}$,  we can solve \eqref{imply} in a full measure subset of the interval
$$
I:=\{ \xi\in \RR:\, \exists \, u\in (0,1), \; \G_\sigma(u)=\xi\}.
$$

Extending $u$ by $0$  if $I $ is bounded below or by $1$ if  $I $ is bounded above, we will  have a function defined a.e. in $\RR$ that is a formal solution of \eqref{pvi}.

\begin{remark} We think that, probably under certain additional hypotheses, it could be proven that $\sigma_s$ is the limit of the speed of propagation for the corresponding viscosity equation as in the regular case $(H_r)$.
\end{remark}

\subsection{Some toy examples.}

 The purpose of this subsection is to show examples where  $\sigma_s<\infty$. The main objective is to get to Example 4, where the existence of saturated profiles with a single jump saturation point is shown. Although the constructed operator does not verify $(H_r)$, a small perturbation argument allows us to modify it and obtain one that meets this condition, as we will see in the next subsection. Examples 1, 2 and 3, while interesting in themselves, have been introduced to make it easier for the reader to understand  how the parameters shown in Example 4 are defined.

 Similar examples of discontinuous profiles with a single point of discontinuity can be found in \cite{CCCSSsurv,CCCSSinv}. In those cases, however, $a^+(0)=0$ and the perturbation argument use regularity at $u=0$ and $u=1$.
 
 \
 
Although we could work with more general operators, we are going to limit ourselves to cases where 
 $$a(u,s)=D(u)\phi(s)$$ defined on $
\Omega=[0,1]\times \RR$ with 
\begin{itemize}
\item the diffusion term $D:[0,1]\to [0,\infty)$ a $C^2-$function and 
\item the flux limiter
$\phi: \RR\to (-1,1)$ a regular function with $\phi(0)=0, \; \phi'(s)>0, \; s\in \RR$ and $\lim_{s\to \pm\infty}\phi(s)=\pm 1$.
 \end{itemize}
For example, we can think of $$\phi(s)=\displaystyle\frac{s}{(1+|s|^p)^\frac 1 p},\; p\geq 2$$ or $$\phi(s)=\frac 2 \pi \arctan (s).$$ Therefore $a_+(u)=D(u)$ and the set of totally degenerated levels is $$L_{td}=\{u\in [0,1]:\, D(u)=0\,\}.$$

In the examples that follow, the role of $\phi$ is not relevant. We may proceed as if we were considering the same in all of them, so that $D$ determines the flux function.

\

 {\bf \noindent Example 1.-} Take $u_1\in (0,1)$ and consider $D_1\in C^2[0,1]$ with  $D_1(u)=0$ when $0\leq u\leq u_1$ and $\ddot{D}_1(u)>0$ if $u_1<u\leq 1$.  We have  $$L_{td}=[0,u_1]$$ and  all the points in $(u_1,1]$ are regular levels.
Therefore, as has been done in the previous sections, it can  be proven that for all $\sigma\geq 0$ there exists $\alpha_\sigma\in [0,1)$ 
so that $0<\V_\sigma(u)<a_+(u)=D_1(u), \; u\in(\alpha_\sigma, 1)$. Moreover, when $\alpha_\sigma> u_1$, $\V_\sigma(\alpha_\sigma)=D_1(\alpha_\sigma)>0$. 

The function $\sigma\in[0,\infty)\to \alpha_\sigma\in[0,1)$ is decreasing and, as in the regular case, a priori we can only guarantee  that it is upper semi-continuous,  i.e., for any $\sigma_0\geq 0$
$$
\limsup_{\sigma\to \sigma_0} \alpha_\sigma\leq \alpha_{\sigma_0}.
$$

\begin{lemma}\label{dot}
Equation $\sigma = \dot {D}_1(\alpha_\sigma)$ has  a unique root  $\tau>0$, and $\alpha_\sigma>u_1$ for any $\sigma\in [0,\tau]$. The map $\sigma\in [0,\tau]\to \alpha_\sigma\in  (u_1,1)$ is continuous. Moreover, $\alpha_\sigma=u_1$ for any $\sigma>\tau$.
\end{lemma}
\proof
Just using monotony we can take a value $\tau\geq 0$, such that $\sigma > \dot {D}_1(\alpha_\sigma)$  when $\sigma>\tau$ and $\sigma < \dot {D}_1(\alpha_\sigma)$ when $\sigma<\tau$.
Moreover, since  $\R_0(u)>0, \; u\in [0,1)$, and $\alpha_0>u_1$, then  $\tau>0$.

When  $\sigma<\tau$, the intersection between $\V_\sigma$ and $D_1$ is transversal and, therefore, $\alpha_\sigma$ is continuous in $\sigma$. Furthermore, the upper semi-continuity plus monotony provides the continuity to the left of $\alpha_\sigma$ at $\sigma=\tau$. Using continuity to the left in $\sigma=\tau$ we obtain that $\tau=\dot {D}_1(\alpha_\tau)>0$.

 Finally, when $\sigma>\tau$ if  $\alpha_\sigma>u_1$, since $\sigma>\dot {D}_1(\alpha_\sigma)$, the function $d(u)=\V_\sigma(u)-D_1(u)$ satisfies $d(\alpha_\sigma)=0$ and $\dot{d}(\alpha_\sigma)>0$.  So, for some $\varepsilon >0$, $\V_\sigma(u)>D_1(u), \; u\in (\alpha_\sigma, \alpha_\sigma+\varepsilon)$, in contradiction with the definition of $\alpha_\sigma$. Hence,  $\alpha_\sigma=u_1$. 

\qed

As a consequence of  Lemma \ref{dot}, if $\sigma>\tau, \, \V_\sigma(u)=0$ for any $u\in [0,u_1]$. In particular,  $\sigma_s\leq \tau$ and it is finite.  We are going to complete a description of
$\V_\sigma(u)$ in the case $\sigma\leq \tau$. Define
$$\beta_\sigma=\alpha_\sigma-\displaystyle\frac{\V_\sigma(\alpha_\sigma)}{\sigma}.$$
$\beta_\sigma$ is the ordinate of the intersection with the horizontal axis of the line through  point $(\alpha_\sigma,\V_\sigma(\alpha_\sigma))$ with slope $\sigma$.
If $ \beta_\sigma\le u_1$, using  the uniqueness of solution and the convexity of $D_1$, we have
$$\V_\sigma(u)=\sigma(u-\alpha_\sigma)+\V_\sigma(\alpha_\sigma),\; \max \{0,\beta_\sigma\}\leq u\leq \alpha_\sigma.$$
Hence, if $\beta_\sigma\ge 0$, $\V_\sigma(u)=0$ for $u\in [0,\beta_\sigma]$  and $\sigma\ge \sigma_s$. When  $\beta_\sigma< 0$, $\V_\sigma(0)>0$ and $\sigma<\sigma_s$. Note that  $\beta_\sigma\to -\infty$ if $\sigma\to 0$, so  $\beta_\sigma< 0$  when $\sigma$ is small enough.

The map $\sigma\in (0,\infty)\to \alpha_\sigma\in [0,1)$ has a  jump discontinuity at $\sigma=\tau$. However, its restriction to
 $(0,\tau]$ is continuous. Therefore, the map $\sigma\in (0,\tau)\to \beta_\sigma$ is continuous and we can obtain $\sigma_s$ by solving $\beta_\sigma=0$ with $\sigma\in (0,\tau)$.

\begin{lemma}\label{mottt}
The map $\sigma \in (0,\tau) \to \beta_\sigma $ is strictly decreasing. 
\end{lemma}

\begin{proof}
 When $\sigma\in (0,\tau)$, $\V_\sigma$ transversely crosses $D_1$ at $u=\alpha_\sigma$ and we can always find $\bar u<\alpha_\sigma$ such that $\V_\sigma(\bar u)>D_1(\bar u).$ So $$\beta_\sigma=\bar u -\displaystyle\frac{\V_\sigma(\bar u)}{\sigma},$$ since 
 $\V_\sigma(u)$ is a line in the interval $(\bar u,\alpha_\sigma)$. This formula allows us to obtain the strict monotonicity from the monotonicity of $\V_\sigma$ on small intervals.
 
\end{proof}

If $\sigma=\tau$,  the line 
$\V=\tau(u-\alpha_\tau)+\V_\tau(\alpha_\tau)$ is tangent to the curve $\V=D_1(u)$ and $\V_\tau$ cannot cross $D_1$ in $(u_1, \alpha_\tau)$ since $\dot{D}_1(u)<\tau$. When   $\sigma<\tau$, the line $\V=\sigma(u-\alpha_\sigma)+\V_\sigma(\alpha_\sigma)$ intersects curve $\V=D_1(u)$ on a second point $u^*\in (\beta_\sigma, \alpha_\sigma)$ and $\V_\sigma(u)=\sigma(u-\alpha_\sigma)+\V_\sigma(\alpha_\sigma),\; u^*\leq u\leq \alpha_\sigma.$ Moreover,   $0<\V_\sigma(u)< D_1(u)$ for $u\in (u_1, u^*).$
In any case, $\V_\sigma(u)=0, \, u\in [0,u_1]$. We leave the details for the reader.

\begin{figure}[h]
\includegraphics[width=1.0\textwidth]{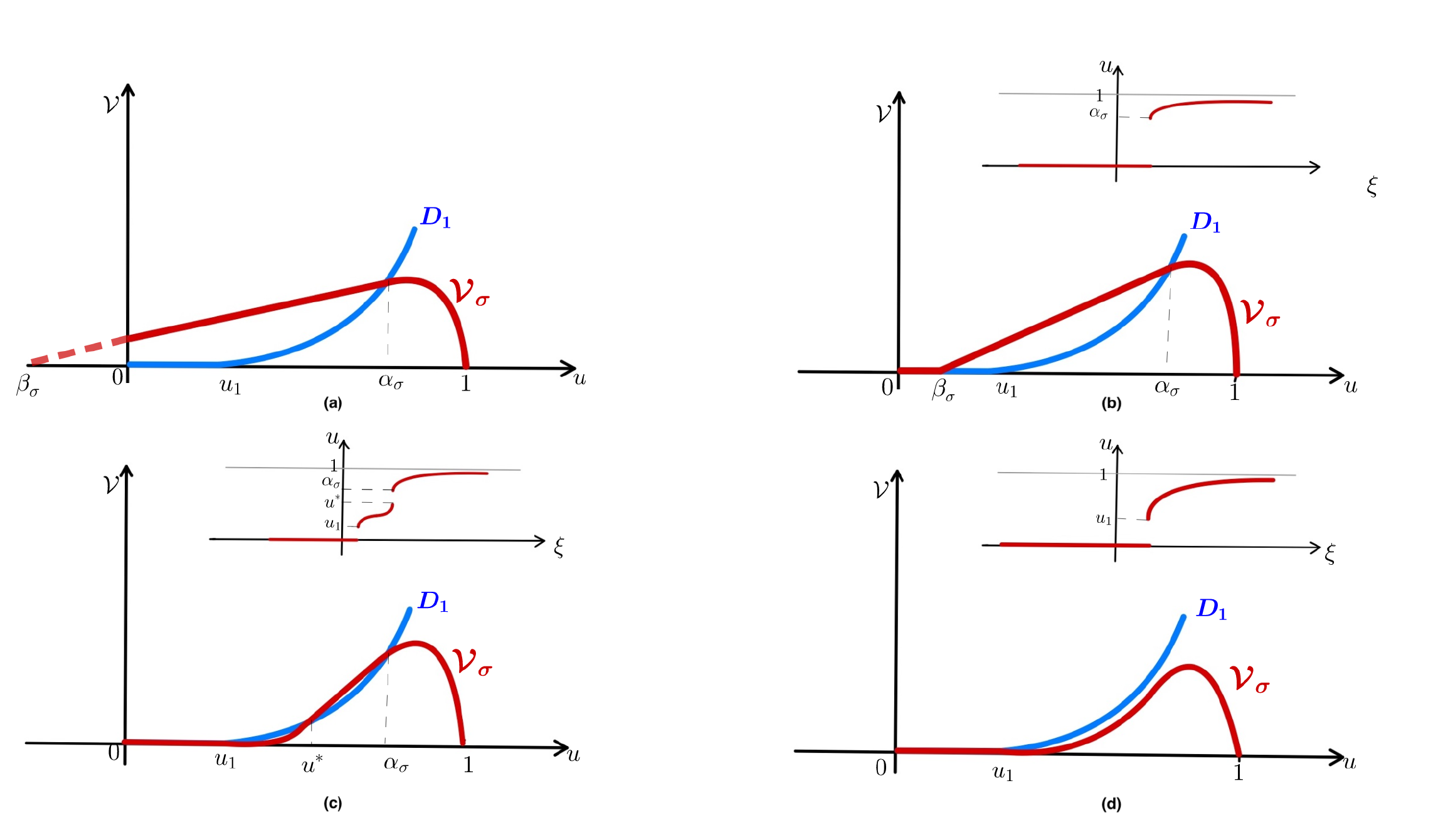} 

\caption{Example 1. \small{(a): $\V_\sigma$ for $\sigma<\sigma_s$.  (b), (c) and 
(d): $\V_\sigma$ for $\sigma\geq \sigma_s$. 
(b) corresponds to $\sigma\leq \tau$ and $0\leq \beta_\sigma\leq u_1$, 
(c) corresponds to $\sigma\leq \tau$ and $u_1<\beta_\sigma$ and 
(d) corresponds to $\sigma> \tau$.  In the small inset, the respective flux-saturated profiles proposed by subsection \ref{ultraframe}}}
\end{figure}

\

{\bf \noindent Example 2.-} Take now $u_2\in (0,1)$ and consider  $D_2\in C^2[0,1]$ with  $D_2(u)=0$ when $u_2\leq u\leq 1$ and $\ddot{D}_2(u)>0$ when $0< u< u_2$. Now,
$$L_{td}=[u_2,1]$$ and all the points in $[0, u_2)$ are regular levels. 
Under these conditions,  $\V_\sigma(u)=0, \; u\in [u_2,1]$ for all  $\sigma\geq 0$. If $u\in (0,u_2)$, 
$\V_\sigma\in C^1((0,u_2))$ and $0<\V_\sigma(u).$ 

\begin{lemma}
 For any $\sigma\geq 0$,
 $$ \V_\sigma(u)<D_2(u), \; u\in (0,u_2).$$
\end{lemma}
\begin{proof}
Since $D_2$ is strictly decreasing in $(0,u_2)$, if $\V_\sigma(u_0)\geq D_2(u_0)$ for some $u_0\in (0,u_2),$ then $\dot{\V}_\sigma(u)=\sigma$ for $u_0<u<1$, in contradiction to $\V_\sigma(u)=0, \, u\in [u_2,1]$.

\end{proof}

Now we argue as in Lemma \ref{l1}, and compute for each $\eta>0$
\begin{equation} \label{sigma0}
  \sigma_0:=\max \{\eta+\frac{f(u)}{g(u,\eta u)}: \, u\in I\}< \infty,
 \end{equation}
where $I=\{u\in (0,u_2): \eta u<D_2(u)\}$, which in this case is an interval  away from $L_{td}$. It is shown as in Lemma \ref{l1} that $\V_\sigma(0)=0$ for all $\sigma\geq \sigma_0$. In particular, $\sigma_s\leq \sigma_0<\infty$.

\begin{figure}[h]
\includegraphics[width=1.0\textwidth]{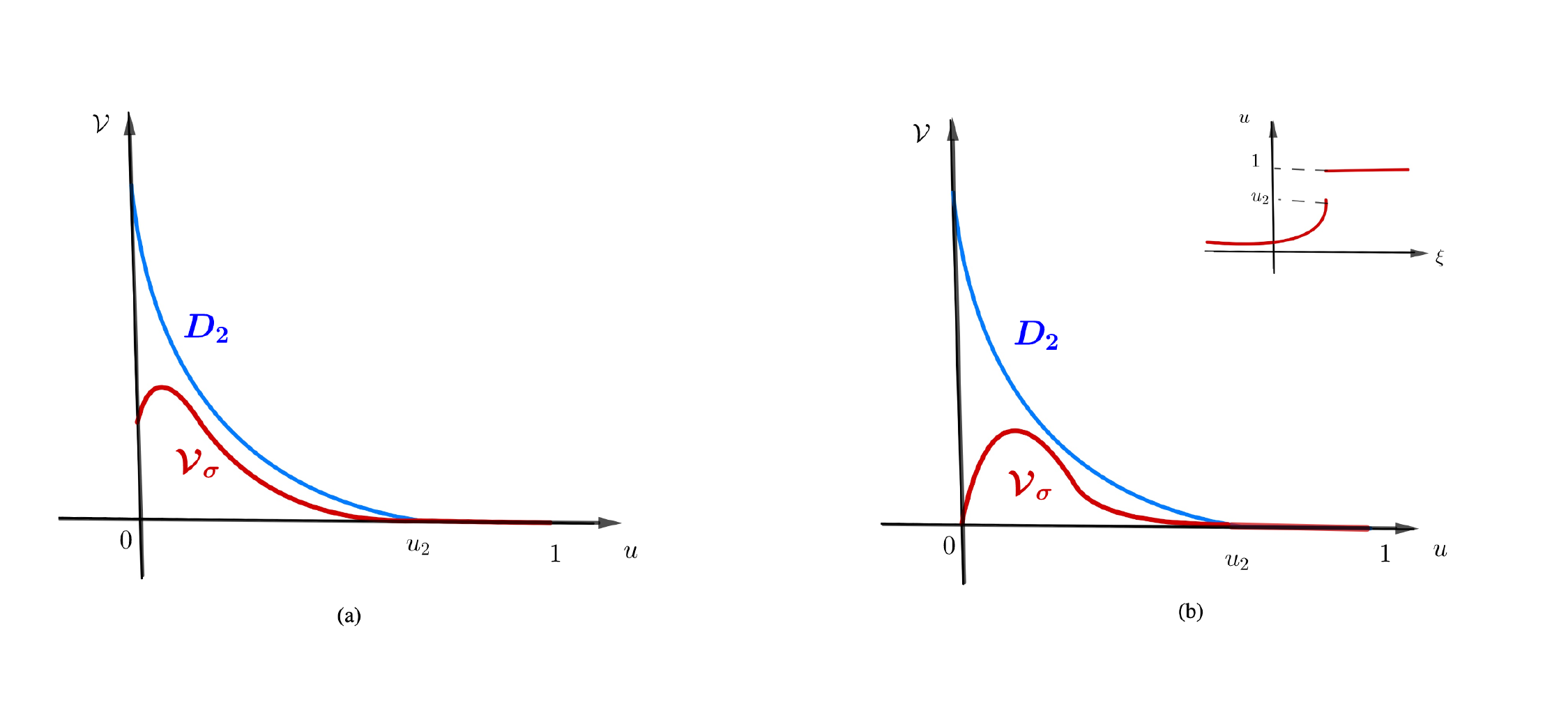} 
\caption{Example 2.  \small{(a): $\V_\sigma$ when  $\sigma<\sigma_s$.  (b): $\V_\sigma$ when  $\sigma>\sigma_s$. In the small inset, the profile of the flux-saturated profile.}}
\end{figure}
\

{\bf \noindent Example 3.-} We are now going to combine the two previous examples. Take $0<u_2<u_1<1$ and define 
$$
 D(u)=\left\{\begin{array}{ll}
 D_2(u),&u\in[0,u_2),\\
 0,&u\in[u_2, u_1],\\ 
 D_1(u),&u\in(u_1,1].
 \end{array}\right.
 $$
Fixing $\sigma\geq 0$, we will denote as $\V_\sigma^{(1)}$ the function determined in Example 1. By Lemma \ref{mottt}, we can obtain a value $\tilde\sigma<\tau$ such that $\beta_{\tilde\sigma}=u_2$. We will also denote as $\sigma_s^{(2)}$ and   $\V_\sigma^{(2)}$ the values of $\sigma_s$  and $\V_\sigma$ obtained in  Example 2.

\

If $\sigma\geq \tilde\sigma$, $\V_\sigma$ has the expression

\begin{equation}\label{caso2}
 \V_\sigma (u)=
\left \{ \begin{array}{ll}
\V_\sigma^{(1)}(u),   &  u\in (\beta_\sigma, 1],\\
0, & u\in [u_2,\beta_\sigma] ,\\
\V_\sigma^{(2)}(u),  & u\in [0,u_2).
\end{array} \right .
\end{equation}

So, if $\sigma\geq \max \{ \tilde\sigma, \sigma_s^{(2)} \}$, $\V_\sigma(0)=0$ and $\sigma_s$ is finite. On the other hand, if $\sigma<\min \{ \tilde\sigma, \sigma_s^{(2)}\}$, then $\beta_\sigma< u_2$ and $\V_\sigma(u)=\V^{(1)}_\sigma (u)$ on a interval $(\gamma_\sigma, 1]$ where $\beta_\sigma<\gamma_\sigma< u_2$ is the abscissa of the point where the line $\V=\sigma(u-\alpha_\sigma)+\V_\sigma(\alpha_\sigma)$ intersects the curve $\V=D_2(u)$. Then $\V_\sigma(u_2)>0$ and so $\V_\sigma(u)\geq \V_\sigma^{(2)}(u)$ for $u\in (0,u_2)$. 
Therefore, $\V_\sigma(0)\geq\V_\sigma^{(2)}(0)>0,$ and $\sigma_s\geq \min \{ \tilde\sigma, \sigma_s^{(2)} \}$.

\begin{figure}[h]
\includegraphics[width=.99\textwidth]{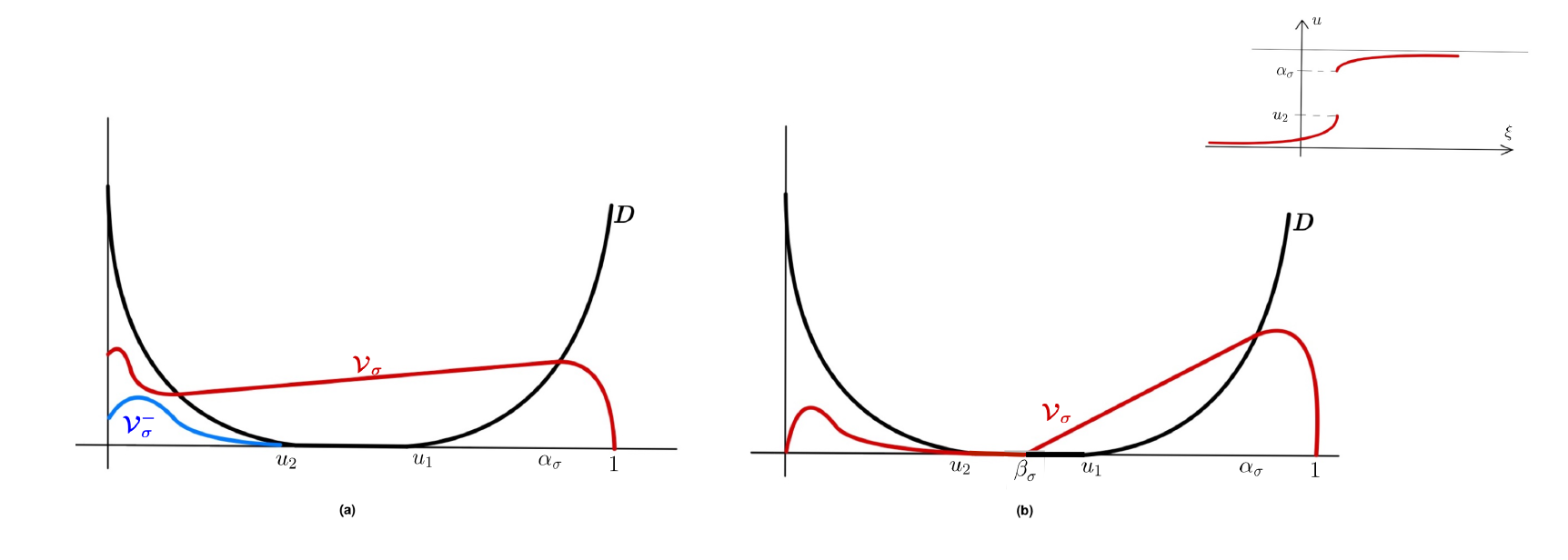} 
\caption{\small{$\V_\sigma$ in Example 3 when (a): $\sigma<\min \{ \tilde\sigma, \sigma_s^{(2)} \}$.  (b):  $\sigma\geq \max \{ \tilde\sigma, \sigma_s{(2)} \}$. In the small inset, the profile of the flux-saturated profile.}}
\end{figure}


{\bf \noindent Example 4.-} Now we will modify Example 3 so that  $\V_{\sigma_s}(u)>0$ for all $u\in (0,1)$ and  there are two values $0<\gamma<\alpha<1$ so that 
$\V_{\sigma_s}(u)>D(u)$ for any $u\in (\gamma,\alpha)$, and $\V_{\sigma_s}(u)<D(u)$ for any $u\in [0,\gamma)\cup (\alpha, 1]$.

Consider $D$ as in Example 3 with $\sigma_s^{(2)}<\tilde\sigma.$
This condition can be obtained by the following procedure:

Fixing $u_1\in (0,1)$ and $D_1$ as in Example 1, we choose $0<u_2<u_1$ and we determine
$\tilde\sigma$. Now we find $D_2$ as in Example 2 so that the value $\sigma_0$ defined by  \eqref{sigma0} verifies $\sigma_0<\tilde\sigma$. Then, $\sigma_s^{(2)}\leq \sigma_0<\tilde\sigma$.

So,   $\sigma_s\in [\sigma_s^{(2)},\tilde\sigma]$. Otherwise,  when $\sigma=\tilde\sigma$, by \eqref{caso2} we have $\V_{\tilde\sigma}(u)=\V^{(2)}_{\tilde\sigma}(u)$ for $u\in [0,u_2]$ and
$\V^{(2)}_{\tilde\sigma}(0)=0$, since $\sigma_s^{(2)}<\tilde\sigma.$ Therefore, $\sigma_s<\tilde\sigma$, so  $\beta_{\sigma_s} <u_2$ and $\V_{\sigma_s}(u)=\V^{(1)}_{\sigma_s}(u)>0$ in $u\in [u_2,1)$. Since $\V_{\sigma_s}(0)=0$, there must be $\gamma$ so that
 $$D_2(\gamma)=\V_{\sigma_s}(\gamma),$$
 and
 $$
 \V_{\sigma_s} (u)=
\sigma_s(u-\alpha_{\sigma_s})+\V_{\sigma_s}(\alpha_{\sigma_s}), \, u\in [\gamma,\alpha_{\sigma_s}].$$
Obviously, $ \V_{\sigma_s} (u)<D_2(u)$ for $u\in [0,\gamma)$ as in Example 3. Taking $\alpha=\alpha_{\sigma_s}$ we have the desired condition.

\subsection{A one-parameter class of regular examples.}

The purpose of this subsection is to show that flux-saturated profiles can appear  for values of $\sigma<\sigma_s$. To do this, we are going to build a one-parameter family of diffusion functions from the function $D$ of Example 4, so that the resulting flux functions satisfy the hypotheses $(H_r)$ and $(H_c)$. 
Thus, their corresponding $\sigma_s$ are the limits when $\varepsilon\to 0$ of the minimum speed of propagation for classic  TW's of the viscosity approximations of the aforementioned reaction-diffusion equations, see  \eqref{v}. Furthermore, the viscosity approximations that appear are uniformly elliptical, in which case everything seems to work fine.

It would appear that the Bertsch-Dal Passo condition \eqref{BPass} prevents this situation; but this question will be analyzed in future works.

\

Consider $a(u,s)=D(u)\phi(s)$ as in Example 4. To simplify the notation, let us denote as $\bar \sigma$ and  $ \varUpsilon(u)$ the corresponding values of $\sigma_s$ and $\V_{\sigma_s}(u)$, respectively,  for this flux function $a$. We know that
$$ \varUpsilon(u)>0, \; u\in (0,1).$$
and that there exist $0<\gamma<u_2<u_1 <\alpha<1$, so that $\varUpsilon(u)>D(u)$ when $u\in (\gamma,\alpha)$ and $\varUpsilon(u)<D(u)$
if $u\in [0,\gamma)\cup (\alpha,1]$.

\

Consider now  $\tilde D :[0,1]\to [0,\infty)$ as a $C^2$-function  so that  the support of $\tilde D$ is a compact set, $[\delta, \kappa]$, with $\gamma<\delta<u_2<u_1<\kappa<\alpha$.

Given $\lambda> 0$, we define $a^\lambda(u,s)=D^\lambda (u)\phi(s)$ where
$$
D^\lambda(u)= D(u)+\lambda \tilde D(u).
$$

\begin{figure}[h]
\includegraphics[width=.7\textwidth]{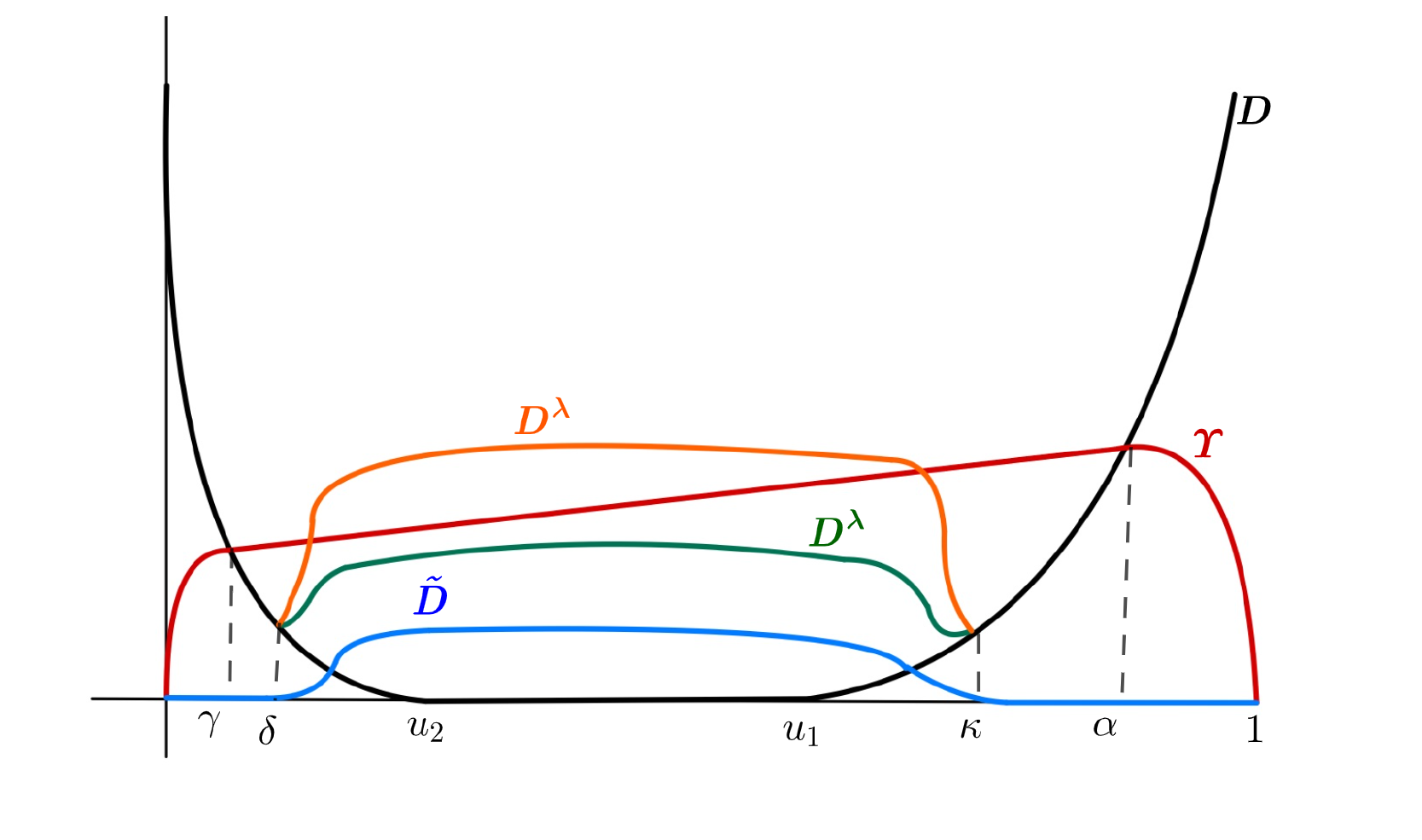}
\caption{Sketch of the graphs of the functions $D$, in black, $\varUpsilon$, in red, $\tilde D$, in blue  and $D^\lambda$ on the compact set $[\delta,\kappa]$ when   $\lambda$ is either small, in green,  or  large, in orange. Outside of $[\delta,\kappa]$, $D^\lambda=D$ for all $\lambda$.}

\end{figure}

It is obvious that  the flux function $a^\lambda(u,s)$ fulfills the hypotheses of Theorems \ref{tp} and \ref{v3} for each $\lambda>0$.
Let us respectively call $\sigma^\lambda_r, \;\sigma^\lambda_s$ and $\V_\sigma^\lambda$ the corresponding values of $\sigma_r, \; \sigma_s$ and the function $\V_{\sigma}(u)$ solution of \eqref{1p} for this flux function $a^\lambda$.

\

When $\lambda$ is small enough,
\begin{equation} \label{se}
 a^\lambda_+(u)=D^\lambda(u)\leq\varUpsilon(u) \mbox{ for any }u\in (\gamma,\alpha).
\end{equation}
 Therefore,  $ \sigma_s^\lambda\leq\bar \sigma$ and $\V^\lambda_{\bar \sigma}(u) =\varUpsilon(u)$.
 \begin{lemma}
  With the notations of Example 3, assume that
\begin{equation}\label{e4}
2 \displaystyle\sqrt{D_2(0)\phi'(0)\dot{f}(0)}< \sigma_s^{(2)}.
\end{equation}
Then $ \sigma_s^\lambda=\bar \sigma$ for any $\lambda$ satisfying \eqref{se}.
 \end{lemma}
\begin{proof}
 It is a consequence of Theorem \ref{caracter}, since  in such a case
 \begin{equation}\label{gammacero}
 \dot {\varUpsilon}(0)=\frac{\bar \sigma}{2} + \displaystyle\sqrt{\frac{\bar \sigma^2}{4}-D_2(0)\phi'(0)\dot{f}(0)}.
\end{equation}
\end{proof}

 \begin{remark}\label{ultimoo}
 When $\lambda$ is small enough for \eqref{se} to hold,  $\sigma^\lambda_s<\sigma^\lambda_r$, since there exist values 
  $u\in (0,1),$ for which  $\V_{\bar \sigma}(u) > a_+^\lambda(u)$; and this also happens when $\sigma$ is near enough to $\bar \sigma$  by continuity. In this case, there is no classic TW moving at speed $\sigma^\lambda_r$ as stated in Remark \ref{dif}.
 \end{remark}
  
Moreover, if for a $\lambda$,  \eqref{se} is fullfiled we have
  $$
  \G^\lambda_{\bar \sigma}(u)=\G_{\bar \sigma}(u), \; u\in (0,1).
  $$
Then,  the corresponding profile $u_{\bar \sigma} : \RR \to (0,1)$, provided by  Theorem \ref{ts} for $a=a^\lambda$ and $\sigma=\bar \sigma$, is singular and does not depend on $\lambda$. The singular set $\mathcal{S}_{\bar \sigma}$ is a singleton and $u_{\bar \sigma}$ satisfies the Rankine-Hugoniot condition.

This shows that even when  $\lambda$ is such that \eqref{se} does not hold, $u_{\bar \sigma}$  is still a singular solution in the sense of Section \ref{A}. The following Lemma completes our construction.

\begin{lemma}
 $\bar \sigma < \sigma_s^\lambda$ when $\lambda$ is large enough.
\end{lemma}
\begin{proof}
Since $\varUpsilon(u)=\bar \sigma u +\bar c, \; u\in [\gamma,\alpha],$ for some $\bar c\in \RR$, when  $\lambda$ is large enough
\begin{equation}\label{fail}
 D^\lambda(u)>\bar \sigma u +\bar c, \; u\in [u_1,u_2].
\end{equation}
So $\varUpsilon(u)$ is not a solution of \eqref{1p} with $a=a^\lambda$.

But  $\V_{\bar \sigma}^ \lambda (u)=\varUpsilon(u)$
when $u\in [\alpha,1]$ because $D^\lambda=D$ on $[\alpha,1]$. Moreover,  $$\V_{\bar \sigma}^ \lambda (u)\geq \bar \sigma u +\bar c, \;u\in [\gamma,\alpha],$$ since $\dot{\V}_{\bar \sigma}^ \lambda (u)\leq \sigma$ and, by \eqref{fail},  the last inequality has to be strict at some point in $[u_1,u_2]$. Then, $\V_{\bar \sigma}^ \lambda (\gamma)> \bar \sigma \gamma +\bar c=\varUpsilon(\gamma)$.

By \eqref{gammacero}, Proposition \ref{pviA3} allows us to affirm that $\V_{\bar \sigma}^ \lambda (0)>0$ and, therefore, $\bar \sigma < \sigma_s^\lambda$ when $\lambda$ is large enough.

\end{proof}

\begin{remark}\label{nu}
 We know that
in the regular case, i.e., when the flux function $a$ verifies $(H_r)$,  the profiles of classic TW solutions of equation \eqref{rd} are unique up to translations of the independent variable, see Remark \ref{2.6}.

The next example shows that
when we deal with singular profiles, the answer is not so simple, even if $a$ verifies $(H_r)$ and $(H_c)$. Here we  sketch  a construction of two singular profiles, $U_1$ and $U_2$, that are not a translations of the independent variable for the same flux and moving at the same speed.
\end{remark}

Now let us take $\hat\sigma>\bar\sigma$ close enough to $\bar\sigma$ so that the function  $ \hat\varUpsilon(u)=\V_{\hat\sigma}(u)$  has a similar shape to $ \varUpsilon(u)$, that is $ \hat\varUpsilon(u)>0, \; u\in (0,1),$
and
 there exist $0<\hat\gamma<u_2<u_1 <\hat\alpha<1$, so that $\hat\varUpsilon(u)>D(u)$ when $u\in (\hat\gamma,\hat\alpha)$ and $\hat\varUpsilon(u)<D(u)$
if $u\in [0,\hat\gamma)\cup (\hat\alpha,1]$ but, by Theorem \ref{caracter}, with
\begin{equation}\label{varU}
 \dot {\hat\varUpsilon}(0)<\frac{\hat \sigma}{2} + \displaystyle\sqrt{\frac{\hat \sigma^2}{4}-D_2(0)\dot{\phi}(0)\dot{f}(0)}
\end{equation}
because of $\hat\sigma>\bar\sigma$.

Taking $\tilde D$ and $D^\lambda$ as before, we have that $\hat\varUpsilon(u)=\V^\lambda_{\hat\sigma}(u)$ for each $\lambda\leq \hat\lambda$, where $\hat\lambda$ is the first positive value for which
$$
D^{\hat\lambda}(u)\leq \hat\sigma u + \bar c, \; u\in (\hat\gamma, \hat\alpha),
$$
and the equality holds for at least some $\hat u \in (\hat\gamma, \hat\alpha)$.

If we denote as $U_1(\xi)$ the function defined by \eqref{imply} when $\V_\sigma=\hat\varUpsilon$,
$U_1$ is a singular solution of \eqref{2o} with $a=a^\lambda$, for each $0\leq \lambda\leq \hat\lambda$. Actually, $U_1$ is a singular profile  of \eqref{2o} with $a=a^\lambda$, for each $ \lambda> 0$, but when $\lambda>\hat\lambda$, $\hat\varUpsilon$ does not satisfy the equation in \eqref{1p}.  Therefore, if  $\tilde \lambda>\hat\lambda$ with $\tilde \lambda$ sufficiently close to it, we have $\sigma_s^{\tilde\lambda}<\hat\sigma$, and there must be another singular profile, $U_2:=u_{\tilde \lambda}$  built from $\V_{\hat \sigma}^{\tilde\lambda}$.

Since $\tilde \lambda>\hat \lambda$,  $\hat\varUpsilon(\hat u)=\hat\sigma \hat u  + \bar c<D^{\tilde \lambda}(\hat u),$ $\hat u$ is on the middle of a jump of $U_1$ and is a regular point of $\V_{\hat \sigma}^{\tilde\lambda}$, so $\hat u$ is in the range of $U_2$.

 \appendix

\section{Solutions of a singular Cauchy problem}\label{B}

We present here some results on solutions of a singular initial value problem of the form
\begin{equation} \label{pviA1}
 \left \{ \begin{array}{l}
\dot V=\sigma - \frac{u}{V}\gamma (u,V), \\
V(0)=0.
\end{array} \right. 
\end{equation}
where $\sigma$ is a positive constant and $\gamma :[0,\delta]^2\to (0,\infty)$ a continuous function.

We say that a continuous function, $V: [0,\varepsilon)\to [0,\infty), \; \varepsilon\leq \delta$,  is a  local solution to the right of (\ref{pviA1}) if $V(0)=0$ and for each $u\in (0,\varepsilon)$, $V(u)>0$, there exists $\dot V(u)$ and it satisfies the differential equation. 

Denoting $\gamma_0=\gamma (0,0)$, one has:

\begin{lemma}\label{B1}
Suppose (\ref{pviA1}) has a local solution to the right, $V$. Then,  $\sigma\geq 2\sqrt{\gamma_0}$. Moreover, there exists $\dot V(0)=w$ and it is a solution of 
\begin{equation}\label{cuadratic}
  w^2-\sigma w + \gamma_0=0.
\end{equation}
\end{lemma}

\proof Let $V$ be a local solution to the right of (\ref{pviA1}) and take  $\rho$ as a solution of the scalar first order ODE
$$
\rho'=-V(\rho), \, \rho\in (0,\varepsilon).
$$
Since $V(u)>0, \; u\in (0,\varepsilon)$, it is clear that $\rho$ is a decreasing function defined in $(\alpha,+\infty)$ and $\rho(t)\to 0$ when $t\to \infty$. Moreover, since $V\in C^1(0,\varepsilon)$ and solves (\ref{pviA1}), writing $\gamma(t)=\gamma(\rho(t), -\rho'(t))$, $$\gamma(t)\to \gamma_0, \; t\to +\infty$$ and  $\rho$ satisfies the linear second order ODE
$$
\rho '' -\sigma \rho' +\gamma(t)\rho=0, \; t\in (\alpha, +\infty).
$$
Hence,  denoting $r(t)=- \frac{\rho'(t)}{\rho(t)}=\frac{V(\rho(t))}{\rho(t)}>0$, it  is a solution of the Riccati equation
\begin{equation}\label{riccati}
  r'(t)=r^2(t)-\sigma r(t)+ \gamma(t),
\end{equation}
and  $r'(t)\geq -\frac{\sigma^2}{4}+\gamma(t)$, which is the minimum of the parabola that defines this Riccati equation.

Therefore, if $\sigma< 2\sqrt{\gamma_0}$ ,  $r'(t)$ has to be positive  when  $t$ is large enough and we have two alternatives:
\begin{enumerate}
\item either $\lim_{t\to\infty}r(t)=\bar r>0$, which is not possible because $\bar r$ should be a root of the equation (\ref{cuadratic}), which has no real roots if $\sigma< 2\sqrt{\gamma_0}$,
\item or $\lim_{t\to\infty}r(t)=+\infty$, but in this case $\frac{r'(t)}{r^2(t)}\to 1$ as $t\to +\infty$ and $\frac{r'(t)}{r^2(t)}\geq C>0$ for $t$ large enough. In particular, $r(t)$ would be greater than the solution of a Ricatti equation whose solutions explode in finite time.
\end{enumerate}
So, if there exists a local solution to the right of (\ref{pviA1}), then $\sigma\geq 2\sqrt{\gamma_0}$. 

\

When $\sigma\geq 2\sqrt{\gamma_0}$, equation (\ref{cuadratic}) has at least a real root. Take $r_0$ above the largest of the roots of  (\ref{cuadratic}), that is,
 $r_0>\frac{\sigma}{2}+\sqrt{\sigma^2-4\gamma_0}.$ Then, $r_0^2-\sigma r_0+\gamma_0>0$
 and there exists $t_0>0$ so that  $ r_0^2-\sigma r_0+\gamma(t)>0, \; t\geq t_0$.
 
Again we have two options: 
\begin{enumerate}
\item either there exist $t_1\geq t_0$ with $r(t_1)\geq r_0$, and then
  $r'(t)>0$ when $t>t_1$ and we would get a contradiction as in the previous case, 
  \item or  $r(t)<r_0, \; t\geq t_o$ and, in particular, it is bounded to the right. 
\end{enumerate}

We are going to show that there exists
$$
\lim_{t\to \infty}r(t).
$$
Otherwise, there would be an increasing sequence, $t_n\to \infty$, and two values, $r_1<r_2\leq r_0$,  so that  $r(t_{2n})\to r_1$ and $r(t_{2n+1})\to r_2$. Let $\bar r\in ]r_1,r_2[$ so that $\bar r^2-\sigma \bar r+\gamma_0\neq 0$, for instance, $\bar r^2-\sigma \bar r+\gamma_0< 0$. The other case is similar. Take $t^*\geq t_0$ so that $\bar r^2-\sigma \bar r+\gamma(t)< 0, \;  t\in (t^*,\infty)$. Then, if $r(t)=\bar r$ for some $t\geq t^*$, necessarily $r'(t)<0$, which contradicts the existence of the sequence with the desired conditions.

\

Once we know there exists $\lim_{t\to \infty}r(t)= w$, taking now a sequence $\{t_n\}$ with $r'(t_n)\to 0$ and taking limits in (\ref{riccati}), we obtain that $w$ is a root of (\ref{cuadratic}). But
\begin{equation}\label{limitw}
w=\lim_{t\to \infty}r(t)=\lim_{t\to \infty} - \frac{\rho'(t)}{\rho(t)}=\lim_{t\to \infty}\frac{V(\rho(t))}{\rho(t)}=\lim_{\rho\to 0}\frac{V(\rho)}{\rho}=\dot V(0).
\end{equation}

\qed

Bearing in mind that for equation
 (\ref{1o}), 
$\gamma_0=\frac{\partial a}{\partial s}(0,0)f'(0)$, the previous result  provides an estimate for $\sigma_s$.

\begin{corollary} \label{r2} 

 Assuming that $(H_c)$ and  $(H_r)$ are fulfilled, then
  $$\sigma_s\geq 2\sqrt{\frac{\partial a}{\partial s}(0,0)\dot{f}(0)}.$$
In the ultra degenerate case, Section \ref{EX}, this condition holds when $0\notin L_{td}$.

\end{corollary}
This estimate is the equivalent of \eqref{le} we spoke about in the Introduction.

\begin{theorem}\label{caracter}
Assume $(H_c)$ and  $(H_r)$, $$\sigma> 2\displaystyle\sqrt{\frac{\partial a}{\partial s}(0,0)\dot{f}(0)},$$ and  $\sigma \geq \sigma_s$. Then $\sigma=\sigma_s$ if and only if
$$
\dot{\V}_\sigma(0)=\frac{\sigma}{2} + \displaystyle\sqrt{\frac{\sigma^2}{4}-\frac{\partial a}{\partial s}(0,0)\dot{f}(0)}.
$$
\end{theorem}
\proof 
By  the monotony of  $\V_{\sigma}$ with respect to $\sigma$, if we denote
$$
w_\sigma:=\lim _{u\to 0} \frac{\V_{\sigma}(u)}{u},
$$
$w_\sigma$  is decreasing with respect to $\sigma$. We are going to show that
$$
\sigma>\sigma_s\Leftrightarrow w_\sigma< \frac{\sigma}{2} + \displaystyle\sqrt{\frac{\sigma^2}{4}-\frac{\partial a}{\partial s}(0,0)\dot{f}(0)}.
$$
If  $\sigma>\sigma_s$ 
$$w_\sigma= \frac{\sigma}{2} + \sqrt{\frac{\sigma^2}{4}-\frac{\partial a}{\partial s}(0,0)\dot{f}(0)}$$  
for some $\sigma>\sigma_s$, taking $\tilde \sigma \in (\sigma_s,\sigma)$ 
 $$w_{\tilde{\sigma}}\geq w_\sigma= \frac{\sigma}{2} + \sqrt{\frac{\sigma^2}{4}-\frac{\partial a}{\partial s}(0,0)\dot{f}(0)}> \frac{\tilde{\sigma}}{2} + \displaystyle\sqrt{\frac{\tilde{\sigma}^2}{4}-\frac{\partial a}{\partial s}(0,0)\dot{f}(0)},$$
which is not possible by Lemma \ref{B1}. Hence, 
$$w_\sigma< \frac{\sigma}{2} + \sqrt{\frac{\sigma^2}{4}-\frac{\partial a}{\partial s}(0,0)\dot{f}(0)}.$$ 

 Let us now see  that if 
\begin{equation}\label{sssol}
 w_\sigma< \frac{\sigma}{2} + \sqrt{\frac{\sigma^2}{4}-\frac{\partial a}{\partial s}(0,0)\dot{f}(0)}.
\end{equation}
then $\V_{\tilde{\sigma}}(0)=0$ for $\tilde{\sigma}$ next to $\sigma$ and, therefore, $\sigma>\sigma_s$.

\
 
Indeed, let $$\frac{\sigma}{2}-\sqrt{\frac{\sigma^2}{4}-\frac{\partial a}{\partial s}(0,0)\dot{f}(0)}<\eta<\frac{\sigma}{2}$$ be fixed. Then
$$
\eta^2- \sigma \eta +\frac{\partial a}{\partial s}(0,0)\dot{f}(0)<0,
$$
and there exists $\delta>0$, so that 
$$
0<\eta <\sigma -2\delta- \frac{\frac{\partial a}{\partial s}(0,0)\dot{f}(0)}{\eta}.
$$
But 
$$
\lim_{u\to 0^+} \frac{f(u)}{g(u,\eta u)} = \frac{\frac{\partial a}{\partial s}(0,0)\dot{f}(0)}{\eta},
$$
thus, fixed $\delta>0$, there exists $\varepsilon>0$ so that
\begin{equation}\label{subsol}
\eta <\sigma -\delta- \frac{f(u)}{g(u,\eta u)}, \qquad  u\in (0,\varepsilon). 
\end{equation}
The above inequality allows us to affirm that the function $V(u)=\eta u$ is a strict sub-solution of (\ref{1o}) when  $\tilde \sigma>\sigma-\delta$. Then, if  $\V_{\tilde \sigma}(u_0)<\eta u_0$ for some
 $u_0\in (0,\varepsilon)$, we necessarily have that $\V_{\tilde \sigma}(0)=0$, so that  $\tilde \sigma\geq \sigma_s$ and, therefore,  $\sigma>\sigma_s$.

 The existence of such a $u_0$  for values of $\tilde \sigma$ close to $\sigma$ is a consequence of the continuous dependence on $\sigma$ of $\V_\sigma$, since, 
by Lemma \ref{B1} and  (\ref{sssol}), $w_\sigma$ has to be the smaller root of  (\ref{cuadratic}) with $\gamma_0=\frac{\partial a}{\partial s}(0,0)\dot{f}(0)$. So, $w_\sigma<\eta$ and, therefore,  $\V_{\sigma}(u)<\eta u$ for values of $u$ close to 0. Thus, $\V_{\tilde \sigma}(u_0)<\eta u_0$ for some $u_0\in (0,\varepsilon)$ if $\tilde \sigma$ is close enough to $\sigma$.

\qed
 
\begin{remark} Theorem \ref{caracter} is still true in the ultra degenerate case when $0\notin L_{td}$ and $\sigma_s<+\infty$.
\end{remark}

Returning to the general environment of (\ref{pviA1}), one has

\begin{proposition}\label{limitesup}
 Suppose $\gamma$ is  Lipschitz-continuous with respect to $V$ and $\sigma>2\sqrt{\gamma_0}$. Then, the local solution to the right of \eqref{pviA1} verifying 
\begin{equation}\label{rst3}
 \dot  V(0)=\frac{\sigma +\sqrt{\sigma^2 -4\gamma_0}}{2}
\end{equation}
exists and it is unique.
\end{proposition}
\proof 
Let us look at uniqueness first. Suppose $V$ is a local solution to the right of  \eqref{pviA1}  and write $T(u)=\frac{V(u)}{u}$,  then
$$
u\dot T=\sigma- T-\frac{\gamma(u,uT)}{T}, \; u\in (0,\varepsilon)
$$
and taking $R=T^2$ we have
\begin{equation}
\frac{u}{2}\dot R=\sigma \sqrt{R}-R-\gamma(u, u \sqrt{R}), \; u\in (0,\varepsilon). \label{sssR}
\end{equation}
Calling $w_+=\frac{\sigma +\sqrt{\sigma^2 -4\gamma_0}}{2}$, by \eqref{rst3} we have $T(0)=w_+$ and  $R(0)=w_+^2>0$. 

Denote $$\Psi(u,R)=\sigma \sqrt{R}-R-\gamma(u, u \sqrt{R}).$$ 
$\Psi(u,R)=\Psi_1(R) -\Psi_2(u,R)$, where $\Psi_1(R)=\sigma \sqrt{R}-R$ and  $\Psi_2(u, R)=-\gamma(u, u \sqrt{R})$. Hence,
 $$\dot\Psi_1(w^2_+)=\frac{\sigma}{2w_+} -1<0\; \mbox{ and  }\; |\Psi_2(u,R_2)-\Psi_2(u,R_1)|\leq uL|\sqrt{R_2}-\sqrt{R_1}|,$$ since, by hypothesis, there exists $L>0$ so that
 $$
 |\gamma(u,V_2)-\gamma(u,V_1)|\leq L|V_2-V_1|, \; u, V_1,V_2 \in [0,\delta].
 $$
So, there exists 
$0<\tilde\varepsilon\leq \varepsilon$,
with $(w_+^2-\tilde \varepsilon, w_+^2+\tilde\varepsilon)\subset (0,\delta)$ and $\tilde \delta>0$ so that
\begin{equation}
 \Psi(u,R_2)-\Psi (u,R_1)\leq -\tilde \delta(R_2-R_1), \label{ssstima}
\end{equation} 
 for all $u\in (0,\tilde \varepsilon)$ and  $w_+^2-\tilde \varepsilon< R_1\leq R_2 < w_+^2+\tilde\varepsilon.$ 
 
 \

 Therefore, if $V_1, V_2$ are two local solutions to the right of  \eqref{pviA1} defined in a common interval, $(0, \varepsilon)$, both verifying \eqref{rst3}; and $R_1$ and $R_2$ the functions defined as before which are solutions of \eqref{sssR} and satisfy $R_1(0)=R_2(0)=w_+^2$,  making $\varepsilon$ smaller if necessary, we can suppose
 $$w_+^2-\tilde \varepsilon<R_1(u), R_2(u)<w_+^2+\tilde\varepsilon, \; u\in (0,\tilde\varepsilon).$$ 
 If, for instance, there exists $u_0\in (0,\tilde\varepsilon)$ so that $R_2(u_0)<R_1(u_0)$, by \eqref{sssR},  $D(u)=R_2(u)-R_1(u)$ verifies
$$\dot D\leq \frac{-2\tilde \delta}{u}D, $$
as long as it is positive. Hence, $$D(u)>D(u_0)>0, \; u\in (0,u_0)$$
in contradiction with $D(0)=0$.

\

To prove the existence, note that
$$
\frac{\sigma}{2}<\sigma-\frac{2}{\sigma}\gamma_0,
$$
so the function $\tilde V(u)=\frac{\sigma}{2}u$ is a sub-solution of   \eqref{pviA1} in a neighborhood of 0.
Hence, for all $0<V_0<\delta$, the unique solution of the regular initial value problem
$$
 \left \{ \begin{array}{l}
\dot V=\sigma - \frac{u}{V}\gamma (u,V), \\
V(0)=V_0,
\end{array} \right. 
$$
satisfies $V(u)>\frac{\sigma}{2}u$ and also $V(u)\leq V_0+\sigma u$, since $\gamma$ is positive.

Taking limit as $V_0\to 0$, we obtain a solution of  \eqref{pviA1} defined in a neighborhood to the right of 0
that verifies $\frac{\sigma}{2}u\leq V(u) \leq \sigma u$ and, by Lemma \ref{B1},  \eqref{rst3}.

\qed

\

To finish, observe that if we denote $$\bar \gamma=\max \{ \gamma (u,V) : (u,V)\in [0,\delta]^2\},$$ having chosen  $\varepsilon<\frac{\sigma \delta }{\bar \gamma}$, every solution of the initial value problem 
\begin{equation}\label{pviA2}
 \left \{ \begin{array}{l}
\dot V=\sigma - \frac{u}{V}\gamma (u,V), \\
V(u_0)=V_0.
\end{array} \right. 
\end{equation}
with $0<V_0<\delta$ and  $0<u_0<\varepsilon$, is defined in $[0,u_0]$. Moreover, 

\begin{proposition}\label{pviA3}
 In the hypotheses of  Proposition \ref{limitesup}, suppose moreover that the corresponding solution of (\ref{pviA1}), $\bar V$,  is defined in $[0,\varepsilon)$. Then
 \begin{itemize}
 \item If $V_0>\bar V(u_0)$, the solution of  (\ref{pviA2}) verifies $V(0)>0$.
 \item When $V_0<\bar V(u_0)$ the solution of (\ref{pviA2}) verifies (\ref{pviA1}) and $$\dot V(0)=\displaystyle\frac{\sigma -\sqrt{\sigma^2 -4\gamma_0}}{2}.$$
\end{itemize}

\end{proposition}

\proof It is a consequence of the uniqueness of solution of the initial value problem \eqref{pviA2} when $V_0\neq 0$ and  in view of Proposition \ref{limitesup}.

\qed

\end{document}